\documentclass[11pt, reqno]{amsart}
\usepackage[utf8]{inputenc}
\usepackage{amsmath}
\usepackage{amsthm}
\usepackage{hyperref}
\usepackage{amssymb}
\usepackage{amsfonts}
\usepackage{caption}
\usepackage{MnSymbol}
\usepackage{circuitikz, booktabs}
\usepackage{tcolorbox}
\usepackage{textcomp}
\usepackage{relsize}
\usepackage{blindtext,amsmath,amsthm,amsfonts, textcomp, mathrsfs, mathtools}
\usepackage[shortlabels]{enumitem}

\usepackage[top=25mm,bottom=25mm,left=27mm,right=27mm]{geometry}
\newtheorem{theorem}{Theorem}[section]
\newtheorem{conjecture}[theorem]{Conjecture}
\newtheorem{corollary}[theorem]{Corollary}

\newtheorem*{theorem*}{Theorem}
\newtheorem{thmx}{Theorem}

\newtheorem*{remark*}{Remark}
\newtheorem*{problem*}{Problem}
\newtheorem*{conjecture*}{Conjecture}
\newtheorem*{question*}{Question}
\newtheorem{lemma}[theorem]{Lemma}
\newtheorem{proposition}[theorem]{Proposition}

\newcommand{\rom}[1]{\uppercase\expandafter{\romannumeral #1\relax}}
\newcommand{\Q}{\mathbb{Q}}
\newcommand{\Z}{\mathbb{Z}}

\newcommand{\rL}{\mathcal{L}}

\newcommand{\rK}{\mathcal{K}}
\newcommand{\rN}{\mathcal{N}}
\newcommand{\rO}{\mathcal{O}}
\newcommand{\R}{\mathbb{R}}
\newcommand{\C}{\mathbb{C}}
\newcommand{\F}{\mathbb{F}}
\newcommand{\rp}{\mathfrak{p}}

\newcommand{\q}{\mathfrak{q}}
\def\house#1{{%
    \setbox0=\hbox{$#1$}
    \vrule height \dimexpr\ht0+1.4pt width .4pt depth \dp0\relax
    \vrule height \dimexpr\ht0+1.4pt width \dimexpr\wd0+2pt depth \dimexpr-\ht0-1pt\relax
    \llap{$#1$\kern1pt}
    \vrule height \dimexpr\ht0+1.4pt width .4pt depth \dp0\relax
}}

\begin{document}

\title[On points of small height in infinite extensions]{On points of small height in infinite extensions}

\author[Anup B. Dixit]{Anup B. Dixit}
\author[Sushant Kala]{Sushant Kala}

\address{Department of Mathematics\\ Institute of Mathematical Sciences (HBNI)\\ CIT Campus, IV Cross Road\\ Chennai\\ India-600113}
\email{anupdixit@imsc.res.in}

\address{Department of Mathematics\\ Institute of Mathematical Sciences (HBNI)\\ CIT Campus, IV Cross Road\\ Chennai\\ India-600113}
\email{sushant@imsc.res.in}
\date{}

\begin{abstract}
    In this paper, we introduce the notion of asymptotically positive infinite extensions of $\Q$, in the spirit of the Tsfasman-Vl\u{a}du\c{t} theory of asymptotically exact families of number fields. For asymptotically positive extensions, we obtain lower bounds on the logarithmic Weil height, establishing the Bogomolov property for a wide range of infinite non-Galois extensions. Our result is in the spirit of the famous theorem of E. Bombieri and U. Zannier on Bogomolov property for totally $p$-adic extensions of type $(e,f)$. Additionally, our theorem can be interpreted as a $p$-adic equidistribution result on conjugates of $\alpha$, resonating with the archimedean equidistribution theorem \`{a} la F. Amoroso-M. Mignotte and Y. Bilu. In the parallel setting of elliptic curves, we derive lower bounds on the canonical height for points on an elliptic curve over asymptotically positive extensions, without any restriction on its reduction type.  In particular, this extends a result of M. Baker and C. Petsche in the context of totally $v$-adic extensions.
\end{abstract}



\maketitle

\section{\bf Introduction}
\medskip

The logarithmic Weil height $h:\overline{\Q}\setminus \{0\}\to \R^+\cup\{0\}$ plays a central role in understanding the arithmetic complexity of algebraic numbers and also induces a partial ordering on algebraic numbers with bounded degree. A well-known theorem of Kronecker characterizes the algebraic numbers $\alpha$ with $h(\alpha)=0$ as precisely the roots of unity. For $\alpha \in \overline{\Q}\setminus\{0\}$ not a root of unity, Lehmer famously conjectured that
$$
h(\alpha) > \frac{c}{\deg(\alpha)}
$$
for an absolute constant $c>0$.  This remains an important open problem in number theory. The fact that $h(2^{1/n})= (\log 2)/n$ shows that this expected lower bound of order $1/\deg(\alpha)$ for $h(\alpha)$ is sharp. Consequently, a crucial avenue of inquiry is to determine interesting subsets of $\overline{\Q}$, on which the height is uniformly bounded below.\\

A subset $S\subseteq \overline{\Q}$ is said to have the \textit{Northcott property (N)} if for any $c>0$, the set
\begin{equation*}
    \big\{\alpha \in S \, | \, \alpha\neq 0 \text{ and } h(\alpha) < c\big\}
\end{equation*}
is finite. Closely related is the \textit{Bogomolov property (B)}. We say that a set $S \subset \overline{\Q}$ satisfies the property (B) if there exists a constant $c > 0$ such that 
\begin{equation*}
    \big\{\alpha \in S \, | \, \alpha\neq 0, \alpha \text{ not a root of unity} \text{ and } h(\alpha) < c\big\}
\end{equation*}
is empty. Clearly, if $S$ satisfies (N), it also satisfies (B).\\

In \cite{Amoroso}, F. Amoroso and R. Dvornicich  proved that $\Q^{ab}$, the maximal abelian extension of $\Q$, satisfies property (B). They showed that for non-zero $\alpha\in \Q^{ab}$, not a root of unity,
\begin{equation*}
    h(\alpha) > \frac{\log 5}{12}.
\end{equation*}
This was generalized to $K^{ab}$ by F. Amoroso and U. Zannier \cite{Amoroso II} for any number field $K$. Earlier, in 1973, A. Schinzel \cite{Schinzel} obtained property (B) for the set $\Q^{tr}$ of totally real numbers. More precisely, for $\alpha \in \Q^{tr} \setminus \{\pm 1 \}$, he proved that
\begin{equation*}
h(\alpha) \geq \frac{1}{2} \log \left( \frac{1+\sqrt{5}}{2} \right).
\end{equation*}
\medskip

Another family of infinite extensions of $\Q$ for which we know property (B) are totally $p$-adic fields of type $(e,f)$, i.e., infinite Galois extensions of $\Q$ with finite local degree over a prime $p$. This is the famous theorem of E. Bombieri and U. Zannier \cite{BZ}.

\begin{thmx}[Bombieri-Zannier]\label{B-Z}
Let $\rL$ over $\Q$ be an infinite Galois extension. Define
\begin{equation*}
    S(\rL):= \bigg\{p \,\text{ prime} \, \big| \, [\rL_{v} : \Q_p] < \infty \,\, \text{for some}\,\, v \,\, \text{above}\,\, p\bigg\}.
\end{equation*}
Then 
\begin{equation}\label{BZ-eq}
    \liminf_{\alpha\in \rL} \,\, h(\alpha) \,\, \geq \frac{1}{2} \, \mathlarger{\mathlarger{\sum}}_{p\in S(\rL)}  \, \frac{\log p}{e_p \, \left( p^{f_p} + 1 \right)},
\end{equation}
where $e_p$ is the ramification index and $f_p$ is the residual degree of $\rL_{v}/\Q_p$.
\end{thmx}
Thus, if $S(\rL)$ is non-empty, then $\rL$ satisfies property (B). Furthermore, if the right-hand side in \eqref{BZ-eq} diverges, then $\rL$ satisfies property (N). An explicit example where the RHS diverges can be found in S. Checcoli and A. Fehm \cite{Checcoli}. Note that if $p\in S(\rL)$, then for any $\alpha\in \rL$, all the conjugates of $\alpha$ lie in a finite extension $L_{v}$ of $\Q_p$. This can be considered as a non-archimedean analog of Schinzel's result on $\Q^{tr}$. There have been refinements of Theorem \ref{B-Z} in recent times. In 2015, P. Fili and C. Petsche \cite{Fili Petshe} used methods of potential theory to obtain lower bounds on the height of algebraic numbers, all whose conjugates lie in certain local fields. This was later extended by P. Fili and I. Pritsker \cite{Fili Pritsker} to obtain new lower bounds for height of algebraic numbers all of whose conjugates lie in real intervals and $p$-adic discs. \\

All the above results give property (B) for infinite Galois extensions of $\Q$. In this paper, we establish property (B) for certain infinite extensions of $\Q$, which are not necessarily Galois. This is a rare instance where property (B) is proved for infinite non-Galois extensions. One of the main challenges in non-Galois extensions is the absence of natural invariants such as $e_p$ and $f_p$, which were crucial in \eqref{BZ-eq}. To overcome this, we introduce the theory of asymptotically positive extensions, inspired by the study of asymptotically exact families by M. A. Tsfasman and S. G. Vl\u{a}du\c{t} \cite{TV}.\\

For a number field $K/\Q$ and a rational prime power $q=p^f$, define
\begin{equation*}
    \rN_q(K) := \text{ the number of prime ideals of } K \text{ with norm } q.
\end{equation*}
An infinite extension $\rL/\Q$ can be written as a tower of number fields
$$
\rL \supsetneq \cdots \supsetneq L_m \supsetneq L_{m-1} \supsetneq \cdots \supsetneq L_1 = \Q,
$$
where $L_i/\Q$ are finite extensions. Define
$$
    \psi_q(\rL) = \psi_q := \lim_{i\to \infty} \frac{\mathcal{N}_q(L_i)}{[L_i:\Q]}.
$$
This limit exists and is well-defined, i.e., independent of the tower $\{L_i\}$ (as shown in Proposition \ref{well-defined}) and $0\leq \psi_q \leq 1$. For instance, if a prime $p$ splits completely in $\rL$, then $\psi_p =1$ and $\psi_{p^f} =0$ for all $f>1$. If $\psi_q > 0$ for some prime power $q$, we call the extension $\rL$ to be \textit{asymptotically positive}.\\ 

Inspired by Theorem \ref{B-Z}, the authors proposed the following conjecture in \cite[Conjecture 2.2]{dixit-kala}.

\begin{conjecture}\label{A-S}
     Let $\rL / \Q$ be an infinite extension of $\Q$. Then,
    \begin{equation}\label{conjecture-SD}
    \liminf_{\alpha \in \rK} h(\alpha) \geq \frac{1}{2} \sum_{q} \psi_q \, \frac{\log q}{q+1},
    \end{equation}
    where $q$ runs over all prime powers. Hence, if $\rL$ is asymptotically positive, then $\rL$ has property (B).
\end{conjecture}

\medskip

For infinite Galois extensions, the RHS in \eqref{conjecture-SD} is same as the RHS in \eqref{BZ-eq}. Indeed, let $\rL/\Q$ be an infinite Galois extension and $p\in S(\rL)$, then 
\begin{equation*}
    \frac{\log p}{e_p (p^{f_p} + 1)} = \frac{\log p^{f_p}}{e_p f_p (p^{f_p } + 1)} = \psi_{p^{f_p}} \, \left( \frac{\log p^{f_p}}{p^{f_p}+1} \right), 
\end{equation*}
with $e_p$ and $f_p$ as before. The advantage of this formulation in terms of $\psi_q$'s enables us to conjecture the analog of Theorem \ref{B-Z} for infinite non-Galois extensions. For totally $p$-adic fields, Bombieri and Zannier \cite[Example 3]{BZ} demonstrated that \eqref{BZ-eq} is close to being optimal. More precisely, let $S$ be a finite set of primes and $\rL$ be the field of all totally $p$-adic algebraic numbers for $p\in S$. Then, they showed that
\begin{equation*}
    \sum_{p\in S} \frac{\log p}{p-1}\geq \liminf_{\alpha \in \rL} h(\alpha) \geq \frac{1}{2} \sum_{p\in S} \frac{\log p}{p+1}.
\end{equation*}
Therefore, we also expect the lower bound in Conjecture \ref{A-S} to be close to optimal.\\ 

In this paper, we partially resolve this conjecture as follows.


\begin{theorem}\label{main-theorem}
Let \( \rL / \mathbb{Q} \) be an infinite extension of \( \mathbb{Q} \), and let \( \mathcal{O}_{\rL} \) be the ring of algebraic integers of \( \rL \).
\begin{enumerate}[(a)]
    \item The following inequality holds:
    \[
        \liminf_{\alpha \in \mathcal{O}_{\rL}} h(\alpha)
        \geq \frac{1}{2} \sum_q \psi_q \frac{\log q}{q},
    \]
where \( q \) runs over all prime powers. 

\medskip

    \item If $\psi_q(\rL) > 0$ for $q=p^f$, there exists a \( \lambda \), depending on \( p \) and \( \rL \), such that
    \[
        \liminf_{\alpha \in \rL \setminus \mu_\infty} h(\alpha)
        \geq \frac{\log q}{p^{\lambda}(q+1)}.
    \]
\end{enumerate}
Therefore, if \( \rL \) is asymptotically positive, then \( \rL \) has property (B). Moreover, if the sum in (a) diverges, then \( \mathcal{O}_{\rL} \) has property (N).
\end{theorem}

\medskip

Although Conjecture \ref{A-S} still remains open, the bound in Theorem \ref{main-theorem} (a) establishes it for  algebraic integers. For algebraic non-integers, we obtain a weaker lower bound in Theorem \ref{main-theorem} (b). Nevertheless, this proves property (B) for asymptotically positive extensions.

\medskip
\noindent

\begin{remark*}
    It is important to note that  asymptotically positive extensions contain finitely many roots of unity. This follows immediately from Theorem \ref{main-theorem} (a). Alternatively, one can see this as follows. Let $\rL$ be an asymptotically positive extension with $\psi_{p^f}(\rL) > 0$. Suppose there exists an infinite tower of cyclotomic subfields of $\rL$, say $\rK=\bigcup_{i=1}^{\infty} K_i$ with $ K_i = \Q(\zeta_{n_i})$. Clearly, $\psi_{p^f}(\rK) > 0$. For sufficiently large $i$, the inertia degree of $p$ in $K_i$ must equal $f$. This implies that $p^f \equiv 1  \bmod n_i$ for all large enough $i$, which is a contradiction.
\end{remark*}
\medskip
\noindent


\bigskip

 We emphasize that Theorem \ref{main-theorem} (a) is not a straightforward generalization of Theorem \ref{B-Z}. When $\rL / \Q$ is Galois, the condition $[\rL_{v} : \Q_p] <\infty$ is equivalent to the existence of a finite Galois extension $L/ \Q$ with $L\subset \rL$ such that all the prime ideals above $p$ in $L$ split completely in $\rL$. However, if $\rL/\Q$ is not Galois, one may not have such eventual complete splitting, yet it is still possible to have $\psi_q(\rL)>0$ for some prime power $q$. In this case, applying Theorem \ref{B-Z} to the Galois closure of $\rL/\Q$ would lead to poor bounds on $h(\alpha)$. We illustrate this with an example below.\\


\noindent
{\bf Example.} Fix a rational prime $p$. We construct an infinite extension $\rL/\Q$ such that $\psi_p > 0$, but $p\notin S(\widetilde{\rL})$, where $\widetilde{\rL}$ is the Galois closure of $\rL$ over $\Q$. Thus, Theorem \ref{B-Z} is not applicable to $\widetilde{\rL}$, where as Theorem \ref{main-theorem} gives property (B) for $\mathcal{O}_{\rL}$.\\

\noindent Let $K/ \Q$ be a number field with distinct prime ideals $P_1, P_2,..., P_s$ and $ Q_1, Q_2,..., Q_t$. For any integer $n \geq 2$, there exists a degree $n$ extension $L$ of $K$ such that $P_1, P_2,..., P_s$ split completely and $Q_1, Q_2,..., Q_t$ totally ramify over $L$. {
Let $L_1/\Q$ be a degree $4$ extension where $p$ splits completely into ${\rp}_1, \rp_2, \rp_3, \rp_4$. Construct $L_2/L_1$ a degree $9$ extension such that $\rp_1, \rp_2, \rp_3$ split completely and $\rp_4$ ramifies, but not totally in $L_2$. This is achieved by first constructing $L_2'/L_1$ of degree $3$  such that $\rp_1, \rp_2, \rp_3$ split completely and $\rp_4$ ramifies completely. Then, we construct $L_2/L_2'$ where all primes in $L_2'$ above $\rp_1, \rp_2$ and $\rp_3$ split completely and the prime above $\rp_4$ also splits completely. Inductively construct $L_{i+1}/L_i$, an extension of degree $p_{i+1}^2$, where $p_i$ is the $i$-th prime number, such that $\rN_p (L_i) - \frac{\rN_p(L_i)}{{p_i}^2 }$ of the prime ideals of $L_i$ with norm $p$ split completely in $L_{i+1}$ and the remaining prime ideals with norm $p$ ramify in $L_{i+1}$, but not completely. For this tower $\rL = \{L_i\}$, note that
$$
\psi_p = \lim_{i\to\infty} \frac{\rN_p(L_i)}{[L_i:\Q]} \geq \prod_{p_i} \left( 1- \frac{1}{{p_i}^2} \right) = \frac{1}{\zeta(2)},
$$
where the product above runs over all rational primes $p_i$. Thus, we have $\psi_p > 0$. However, since the ramification degree of $p$ is unbounded, $p \notin S(\widetilde{\rL})$.\\

\medskip
In special cases, where there is minimal ramification of a prime in the infinite extension, it is feasible to obtain finer lower bounds, which could potentially surpass Theorem \ref{main-theorem} (a). In this context, we define an infinite extension $\rL/\Q$ to be \textit{almost totally split at a prime} $p$ if $\psi_p =1$ and \textit{almost unramified asymptotically positive at a prime $p$} if
\begin{equation*}
    \mathlarger{\mathlarger{\sum}}_{{f\geq 1}} \, f \, \psi_{p^f} = 1.
\end{equation*}

For such extensions, we obtain the following. 

\begin{proposition}\label{almost-split}
    Let $S$ denote a set of rational primes. Let $\rL/\Q$ be an infinite extension which is almost totally split at all primes $p\in S$. Then
    \begin{equation*}
    \liminf_{\alpha \in \rL} h(\alpha) \geq \frac{1}{2} \, \sum_{p\in S} \frac{\log p}{p+1}.
    \end{equation*}
\end{proposition}

In other words, for the almost totally split case, we obtain Theorem \ref{main-theorem} (a) over all algebraic numbers and not just algebraic integers. In the special case of the totally $p$-adic extension $\rL/\Q$, which is the maximal extension of $\Q$ where $p$ splits completely, Proposition \ref{almost-split} gives a slightly weaker result than the best known result obtained by L. Pottmeyer in \cite{pottmeyer}. He showed that if $\alpha$ is totally $p$-adic and not a $(p-1)$-th root of unity, then
\begin{equation}\label{Pottmeyer}
    h(\alpha) \geq \frac{\log (p/2)}{p+1}.
\end{equation}
However, if more than one prime split completely in $\rL$, then Pottmeyer's method may yield weaker bounds than Proposition \ref{almost-split}.\\

\medskip

\noindent We also establish a refined lower bound for almost unramified asymptotically positive extensions. 
\begin{theorem}\label{almost-unramified}
    Let $\rL/\Q$ be an infinite extension which is almost unramified asymptotically positive at a prime $p$. Then
    \begin{equation*}
    \liminf_{\alpha \in \rL} \, h(\alpha) \geq \frac{1}{2}\,\mathlarger{\mathlarger{\sum}}_{{f\geq 1}} \, f^3 \, \left(\psi_{p^f}\right)^2 \frac{\log p}{p^f}.
    \end{equation*}
\end{theorem}
\medskip

\begin{remark*}
    Using the same method, Proposition \ref{almost-split} and Theorem \ref{almost-unramified} can be extended to the infinite extensions $\rL/K$, where $K$ has class number one. In this case, let $S$ be the set of prime ideals $\mathfrak{p}$ in $\mathcal{O}_K$ which almost totally splits in $\rL$. Then
    $$
    \liminf_{\alpha \in \rL} h(\alpha) \geq \frac{1}{2} \sum_{\mathfrak{p} \in \mathcal{O}_K} \frac{\log {\rm Norm}(\mathfrak{p})}{{\rm Norm}(\mathfrak{p}) + 1}.
    $$
    Equivalently, let $\rp$ be a prime ideal in $K$ which is almost unramified in $\rL$, a similar statement as Theorem \ref{almost-unramified} can be deduced.
\end{remark*}
\medskip

It is worth highlighting that the proof of Theorem \ref{B-Z} by Bombieri-Zannier involves obtaining lower bounds on the discriminant of the minimal polynomial of $\alpha$ and using Mahler's inequality (see Theorem \ref{mahler}). Applying the same method directly fails to give any meaningful result in the non-Galois setting. The  novelty in our proof is building the theory of relative heights and establishing a relative Mahler's inequality, which helps in circumventing the obstructions that are encountered otherwise.\\

Another family of algebraic numbers satisfying property (B) can be obtained from the angular equidistribution theorem by Amoroso-Mignotte \cite{Amoroso-Mignotte} (also see M. Mignotte \cite{mignotte} and Y. Bilu \cite{Bilu}).  

\begin{thmx}[\cite{masser-book}, Theorem 15.2]\label{equidistribution}
    For $\alpha \in \overline{\Q}$, let $d = [\Q(\alpha):\Q]$. For any $\theta$ with $0\leq \theta \leq 2\pi$ the number $n$ of conjugates of $\alpha$ in any fixed sector, based at the origin, of angle $\theta$ satisfies
    \begin{equation*}
        \left | n - \frac{\theta}{2\pi} d\right| \leq 24 \left( d^{2/3} \left(\log 2d\right)^{1/3} + d h(\alpha)^{1/3}\right).
     \end{equation*}
\end{thmx}
\noindent In other words, the set of all $\alpha$ whose conjugates do not have angular equidistribution satisfies property (B). If $\rL$ is asymptotically positive, for all but finitely many $\alpha \in \rL$, a positive proportion of their conjugates lie in a finite local extension $L_{v}/\Q_p$. Hence, Theorem \ref{main-theorem} can be regarded as a $p$-adic equidistribution result in the spirit of Theorem \ref{equidistribution}. 
\medskip

\subsection{Lower bound for canonical height of points on elliptic curves}
Let $E$ be an elliptic curve defined over a number field $K$ given by the Weierstrass equation 
$
y^2 = x^3 + A\,x + B,
$ where $A, \, B \in K$.
For any extension $L/ K$, denote by $E(L)$ the set of points $P \in E(\overline{K})$ such that $x(P),  \, y(P) \in L$. Define the logarithmic Weil height of $P$ as
\begin{equation*}
    h(P) := h(x(P)).
\end{equation*}
Although the logarithmic Weil height is easy to define, it is not compatible with respect to the group operation on elliptic curves. With this in mind, N\'{e}ron and Tate introduced a height function, namely the canonical height, defined as
\begin{equation*}
    \widehat{h}(P) = \lim_{n \rightarrow \infty} \frac{h([2^n] \cdot P)}{4^n}.
\end{equation*}
The canonical height $\widehat{h}$ is a quadratic form on $E(\overline{K})$.
\medskip

In the case of elliptic curves, torsion points can be thought of as analogs of roots of unity. Furthermore, $\widehat{h}(P)=0$ if and only if $P$ is a torsion point. Thus the analog of Kronecker's theorem holds for canonical height. It is then natural to consider lower bounds on $\widehat{h}$ for non-torsion points on elliptic curves. In this aspect, Lehmer's conjecture has been formulated in \cite{Anderson and Masser}.\\

A set $S$ of points in $E(\overline{K})$ is said to have the Bogomolov property (B) if there exists a constant $c>0$ such that for all non-torsion points $P\in S$
\begin{equation*}
    \widehat{h}(P) \geq c.
\end{equation*}
Let $\rL/\Q$ be an asymptotically positive infinite extension. By Theorem \ref{main-theorem}, we know that $\rL$ satisfies property (B). It is natural to ask whether the same holds in the case of $E(\rL)$ when $K\subset \rL$. We answer this in the affirmative below.

\begin{theorem}\label{elliptic-analogue}
    Let $E/K$ be an elliptic curve. Let $\rL/\Q$ be an asymptotically positive extension containing $K$. Then $E(\rL)$ satisfies property (B).
\end{theorem}

For a totally $p$-adic extension $\rL$ of type $(e,f)$, M. Baker and C. Petsche \cite{Baker-2} established property (B) for $E(\rL)$ under the condition that $E$  has semi-stable reduction at the prime $p$. We establish property (B) for $E(\rL)$ over asymptotically positive extensions independent of the reduction type of $E$ at $p$.\\

\medskip

This theme of lower bounds on canonical height of elliptic curves  is discussed in Section \ref{canonical_height} and an explicit lower bound for $\widehat{h}(P)$ is obtained in Theorem \ref{Elliptic Case I}.\\

\medskip

The paper is organized as follows. In Section \ref{prelims}, we recall basic definition and some necessary results towards the proof of our main theorems. We develop the theory of relative height and deduce a relative Mahler's inequality in Section \ref{relative-height}. Basic properties of asymptotically positive extensions are obtained in Section \ref{asymp-positive} and we prove Theorem \ref{main-theorem} in Section \ref{APE}. In Section \ref{ATS-AU}, we prove Proposition \ref{almost-split} and Theorem \ref{almost-unramified}. We describe the analog of this problem for elliptic curves in Section \ref{canonical_height}. Developing the necessary tools for local heights in Section \ref{local-height}, we prove Theorem \ref{Elliptic Case I} and hence Theorem \ref{elliptic-analogue} in Section \ref{proof-elliptic}.
\bigskip

\section{\bf Preliminaries}\label{prelims}
\medskip

Let $K/\Q$ be a number field. For $\alpha \in K^*$, the absolute logarithmic height or logarithmic Weil height is defined as
\begin{equation*}
    h(\alpha) = \sum_{v \in M_K} \log^+ |\alpha|_{v},
\end{equation*}
where $M_K$ is the set of all places of $K$, $\log^+ x = \max( 0, \log x)$ and $|\alpha|_{v}$ is the normalized valuation on $\alpha$ defined as:
\begin{equation*}
    |\alpha|_{v} := 
    \begin{cases}
        \left( N\mathfrak{p}\right)^{- \frac{ord_{\mathfrak{p}}(\alpha)}{[K \,: \, \Q]}}, & \text{ if } v \text{ is non-archimedean corresponding to the prime ideal } \mathfrak{p}, \\
        | \sigma(\alpha)|^{\frac{[K_{v}:\R]}{[K \,: \,\Q]}}, & \text{ if } v \text{ is archimedean corresponding to the embedding } \sigma \text{ of } K.\\
    \end{cases}
\end{equation*}

One can also define the logarithmic Weil height through the Mahler measure. Let $f(x) = a_n x^n + \cdots + a_1 x + a_0 \in \Z[x]$ be the minimal polynomial of $\alpha$. Then, its Mahler measure is defined as
\begin{equation*}
    M(\alpha) := |a_n| \prod_{i} \max(1,|\alpha_i|),
\end{equation*}
where $\alpha_i$'s denote the Galois conjugates of $\alpha$. Mahler measure is connected to the height of $\alpha$ by the relation
\begin{equation*}
    \log M(\alpha) =  h(\alpha) \, [\Q(\alpha): \Q].
\end{equation*}
Thus, Lehmer's conjecture can be stated as
\begin{equation*}
    \log M(\alpha) \geq c
\end{equation*}
for all non-zero $\alpha \in \overline{\Q}$ that are not roots of unity, where $c > 0$ is an absolute constant. \\

While this conjecture remains unresolved, there has been significant progress in recent times. The interested reader may refer to the excellent survey articles \cite{Smyth}, \cite{Verger-Gaugry} and the book by D. Masser \cite{masser-book} for a comprehensive account of this problem.\\

We now recall the classical Mahler's inequality \cite{Mahler}.
\begin{theorem}[Mahler]\label{mahler}
    Let $f(x) = a_n x^n + \cdots + a_1 x + a_0 \in \C[x]$ be a polynomial with roots $\alpha_1, \alpha_2, \cdots, \alpha_n$. Let
    $$
        D:= a_n^{2n-2} \, \prod_{i>j} (\alpha_i - \alpha_j)^2
    $$
    be its discriminant. Then, 
    $$
        |D| \leq n^n M(f)^{2n-2},
    $$
    where $M(f)$ denotes the Mahler measure of $f$, given by
    $ M(f) := |a_n| \prod_{i} \max(1,|\alpha_i|)$.
\end{theorem}
We also recall the following well known lemma from algebraic number theory (see \cite[Prop. II. 8.2]{Neukirch})

\begin{lemma}\label{lemma-2}
    Let $K$ be a number field. Suppose $L=K(\alpha)$ and $f_K(x) \in \rO_K[x]$ be the minimal polynomial of $\alpha$ over $K$. For a prime ideal $\rp \in K$, if $$\rp\rO_L = \q_1^{e_1} \q_2^{e_2}\cdots \q_g^{e_g}$$ 
    and $f_K(x)$ factors in $K_{v}$, the local field corresponding to $\rp$, as $$f_K(x)= f_1(x)f_2(x)\cdots f_G(x),$$ 
    then $g=G$ and up to ordering, $\deg f_j(x) = e_j f_j$, where $e_j$ and $f_j$ are the ramification index and the residue class degree of $\q_j$ respectively.
\end{lemma}

In particular, let $L= \Q(\alpha)$ and $f(x) \in \Z[x]$ be the minimal polynomial of $\alpha$. If ${p\rO_L = \rp_1^{e_1}\rp_2^{e_2}\cdots \rp_g^{e_g}}$ and $f(x)$ factors in $\Q_p$ as $f(x) = f_1(x) f_2(x) \cdots f_G(x)$, then $g=G$ and up to ordering $\deg f_j(x) = e_jf_j$.

\bigskip

We also need the following acceleration lemma as in \cite[Lemma 2.1]{Amoroso-0}.\\

\begin{lemma}\label{acc-lemma}
    Let $L$ be a number field, $v$ be a finite place of $K$ over a rational prime $p$ and let $\rho>0$. Let $\gamma_1, \gamma_2 \in \mathcal{O}_L$ such that $\left|\gamma_1-\gamma_2\right|_v \leq p^{-\rho}$. Then for any non-negative integer $\lambda$ we have $\left|\gamma_1^{p^\lambda}-\gamma_2^{p^\lambda}\right|_v \leq p^{-s_{p, \rho}(\lambda)}$ with $s_{p, \rho}(\lambda) \rightarrow+\infty$ for $\lambda \rightarrow+\infty$. More precisely, let us define an integer $k=k_{p, \rho}$ by $k=0$ if $(p-1) \rho>1$ and by

$$
p^{k-1}(p-1) \rho \leq 1<p^k(p-1) \rho
$$
\noindent
otherwise. Then we can take

$$
s_{p, \rho}(\lambda)=p^k \rho+\max (0, \lambda-k) .
$$

\end{lemma}

\bigskip
\bigskip
\bigskip

\medskip

\section{\bf Relative Weil height and Mahler's inequality}\label{relative-height}
\bigskip

In this section, we develop the theory of relative Weil height and prove an analog of Mahler's inequality (Theorem \ref{mahler}), which shall play a crucial role in the proof of Theorems \ref{main-theorem} and \ref{almost-unramified}.\\

Let $K$ be a number field and $\alpha \in \overline{\Q} \, \setminus \, K$ be an algebraic integer of degree $n=[\Q(\alpha):\Q]$. Let 
$$
f(x) = x^n + a_{n-1} x^{n-1} + \ldots + a_0
$$
be the minimal polynomial of $\alpha$ over $\Q$ and $\alpha_1, \alpha_2, \ldots, \alpha_n$ be the distinct complex roots of $f$. Let $f_{K}(x) = x^m + b_{m-1} x^{m-1} + \ldots +b_0$ be the minimal polynomial of $\alpha$ over $O_K$ of degree $m=[K(\alpha) : K]$ and $\alpha_j$'s be its roots in the algebraic closure of $K$. We define the \textit{relative Mahler measure} of $\alpha$ with respect to an embedding  $\sigma : K \hookrightarrow \C$ as

$$
M_{\sigma(K)}(\alpha) = \prod_{i=1}^{m}  \operatorname{max} \left\{ 1, \, |\alpha_{j}^{\sigma}| \right\},
$$
where $\left\{\alpha_j^{\sigma} \, : \, 1 \leq j \leq m \right\}$ are the complex roots of $\sigma(f_K(x))$. Counting all the roots with multiplicity, one obtains
\begin{equation}\label{minimal}
{f_K(x)}^{[K(\alpha) : \Q(\alpha)]} = \prod_{\sigma} \sigma(f_K(x)),
\end{equation}
where $\sigma$ runs over all the embeddings of $K$.

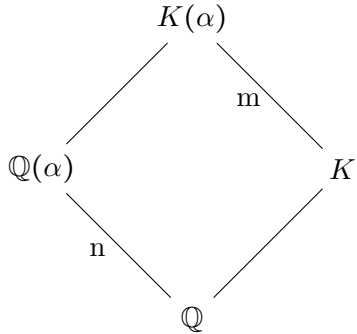
\[
  \begin{tikzpicture}   \node (Q1) at (0,0) {$\Q$};
    \node (Q2) at (2,2) {$K$};
    \node (Q3) at (0,4) {$K(\alpha)$};
    \node (Q4) at (-2,2) {$\Q(\alpha)$};

    \draw (Q1)--(Q2);
    \draw (Q1)--(Q4) node [pos=0.7, below,inner sep=0.25cm] {$n$} ;
    \draw (Q3)--(Q4) ;
    \draw (Q2)--(Q3) node [pos=0.7, below,inner sep=0.25cm] {$m$};
  \end{tikzpicture}
\]

\noindent
Thus, a natural way to define the \textit{relative height} of $\alpha$ with respect to an embedding  $\sigma : K \hookrightarrow \C$ is

$$
h_{\sigma(K)}(\alpha) = \frac{\log M_{\sigma (K)}(\alpha) } {[K(\alpha):K]} = \frac{\log M_{\sigma (K)}(\alpha) }{m}.
$$

\medskip

\noindent
Using the multiplicativity of Mahler measure, from \eqref{minimal} we have

\begin{equation}\label{decomposition}
 h(\alpha) = \frac{1}{[K : \Q]} \mathlarger{\mathlarger{\sum}}_{\sigma} \,  h_{\sigma(K)}(\alpha).
\end{equation}

\medskip

\noindent In other words, the logarithmic Weil height is the average of relative heights with respect to all the embeddings.

\subsection{Mahler's inequality for relative height} Let $\alpha$ and $f_K$ be as above. The discriminant of $f_K$ is given as
$$
D(f_K) = \prod_{i < j\leq m} \, \left( \alpha_i - \alpha_j \right)^2.
$$
Thus for an embedding $\sigma: K \hookrightarrow \C$, we have
\begin{equation*}
    |\sigma(D(f_K))| =  \prod_{i < j\leq m} \, |\sigma(\alpha_i) - \sigma(\alpha_j)|^2.
\end{equation*}
\noindent
We now proceed as in the proof of Mahler's inequality \cite{Mahler}. The product in the last equality is square of absolute value of a Vandermonde determinant. Thus applying Hadamard’s inequality, we obtain

\begin{align*}
    |\sigma(D(f_K))| &\leq \prod_{i=1}^m \, \left( \, \sum_{j=0}^{m-1} |{\sigma(\alpha_i)}^j|^2 \, \right) \\
    & \leq \prod_{i=1}^m \, m \, \operatorname{max} \{1, |{\sigma(\alpha_i)}|^{2m-2} \}\\
    & = m^m M_{\sigma(K)}(\alpha).
\end{align*}
\noindent
Taking the logarithm, we deduce Mahler's inequality for relative height as below.
\begin{lemma}\label{relative-mahler}
    Let $\alpha \in \overline{\Q}\setminus K$ be an algebraic number, with minimal polynomial $f_K(x)$ over $K$ of degree $m$. Then, for any embedding $\sigma: K \hookrightarrow \C$, we have
    \begin{equation*}
        \log |\sigma(D(f_K)| \,\, \leq \,\, m \log m + (2m-2) \,  \log M_{\sigma (K)}(\alpha).
    \end{equation*}
\end{lemma}
In other words, for relative height
$$
h_{\sigma(K)}(\alpha) \geq \frac{1}{2} \left( \frac{\log |\sigma(D(f_K)|}{m^2} - \frac{\log m}{m} \right).
$$

\noindent
Summing over all embeddings of $K$ and using \eqref{decomposition}, we deduce that

\begin{equation}\label{localglobal}
    h(\alpha) \geq \frac{1}{2 \,  [K : \Q]} \left( \frac{\log \left( | N_{K / \Q} \left( D(f_{K}) \right) | \right)}{m^2} \right) -  \frac{\log m}{m} \,.
\end{equation}

\medskip

 It is sometimes advantageous to work with the norm of the discriminant over a number field $K$, as opposed to considering the discriminant over $\Q$. This is mainly because of the identity
\begin{equation*}
    D(f) = |disc(K/\Q)|^2 \, N_{K/\Q} (D(f_K)),
\end{equation*}
where $f$ and $f_K$ are the minimal polynomials of an algebraic number $\alpha$ over $\Q$ and $K$ respectively, and $disc(K/\Q)$ is the discriminant of $K/\Q$. Thus $|N_{K/\Q}(D(f_K))|$ is smaller than $|D(f)|$ and hence, it is more feasible to obtain upper bounds for this norm. 

\medskip

The lower bound in \eqref{localglobal} can be compared with a related result due to Silverman \cite[Theorem 2]{SilvermanIII}, which states that
$$
h(\alpha) \geq \frac{1}{2 m-2} \left( \frac{1}{m} \log |N_{K/\Q} \left(\Delta_{K(\alpha) / K}\right)|-[K : \Q] \log m\right),
$$
where $\Delta_{K(\alpha) / K}$ denotes the discriminant ideal. Notably, when $K = \Q$ and $\alpha$ is an algebraic integer, the relation $|D_f| = [\rO_{\Q(\alpha)} : \Z[\alpha]]^2 |\Delta_{\Q(\alpha)}|$ demonstrates that \eqref{localglobal} provides a sharper lower bound.

\medskip
\begin{remark*}
    The natural obstruction to extending the theory of relative Weil height to algebraic numbers is that the ring of integers of a general number field is not necessarily a principal ideal domain (PID). Consequently, there is no canonical choice for the minimal polynomial of an element $\alpha \in L \setminus K$ over $\mathcal{O}_K$. One can still extend the theory to the case when the number field $K$ has class number one.
\end{remark*}

\bigskip

\section{\bf Asymptotically positive extensions}\label{asymp-positive}
\bigskip

In this section, we state and prove some basic properties of asymptotically positive extensions and the invariant $\psi_q$. \\

\begin{lemma}\label{lemma-1}
Let $L/K$ be an extension of number fields and $p$ be a rational prime. Then,
$$            \mathlarger{\mathlarger{\sum}}_{p^k \leq x}  \frac{\mathcal{N}_{p^k}(K)\, \log p^k}{[K:\Q]} \geq \mathlarger{\mathlarger{\sum}}_{p^k \leq x} \frac{\mathcal{N}_{p^k}(L)\, \log p^k}{[L:\Q]}.
$$
\end{lemma}
\begin{proof}
If a prime ideal $\q$ in $L$ has norm $\leq x$, then the prime ideal below, $\rp = \q \cap K$ also has norm $\leq x$. Suppose a prime ideal $\rp$ in $K$  splits into $\{\q_1, \q_2, \cdots, \q_r\}$ in $L$. Then, 
$$
    \prod_{i\leq r} N(\q_i) = N(\rp)^{[L \, : \, K]}.
$$
Taking log on both sides
$$
    \mathlarger{\mathlarger{\sum}}_{m=1}^n  \,\, \mathcal{N}_{p^m}(L) \log p^m \leq [L : K] \mathlarger{\mathlarger{\sum}}_{m=1}^n  \,\,  \mathcal{N}_{p^m}(K) \log p^m. 
$$
Dividing by $[L:\Q]$, we obtain the lemma.

\end{proof}

Throughout this paper, for a number field $L/\Q$, denote by $n_L:=[L:\Q]$. From Lemma \ref{lemma-1}, we deduce that for any tower $\rL =\{L_i\}$ and any $x > 1$, the limit 
$$
\lim_{i\to\infty} \mathlarger{\mathlarger{\sum}}_{p^k\leq x} \frac{\mathcal{N}_{p^k}(L_i)\, \log p^k}{n_{L_i}}
$$
exists. Therefore, inductively, we can conclude that the limit
$$
\psi_q = \lim_{i\to \infty} \frac{\mathcal{N}_q(L_i)}{n_{L_i}}
$$
exists for all prime powers $q$ and takes value between $0$ and $1$. We now show that this invariant is independent of the choice of tower and hence, is well-defined.

\begin{proposition}\label{well-defined}
    Let $\rL/\Q$ be an infinite extension and $q$ be a rational prime power. Suppose $\{L_i\}$ and $\{K_i\}$ are two towers of number fields such that $\rL = \bigcup_i K_i = \bigcup_i L_i$. Then,
    \begin{equation*}
        \lim_{i\to \infty} \frac{\rN_q(L_i)}{n_{L_i}}= \lim_{i\to\infty}\frac{\rN_q(K_i)}{n_{K_i}}.
    \end{equation*}
\end{proposition}
\begin{proof}
    Let $\psi_q := \lim_{i\to\infty}\frac{\rN_q(K_i)}{n_{K_i}} $ and $\psi_q':=\lim_{i\to \infty} \frac{\rN_q(L_i)}{n_{L_i}}$ for a prime power $q=p^m$. Suppose $K_i = \Q(\alpha)$. Since $K_i\subset \rL = \bigcup_j L_j$, there exists $L_j$ such that $\alpha\in L_j$. Hence, $K_i \subset L_j$ for some $j$. Similarly, every $L_j \subset K_l$ for some $l$. Therefore, applying Lemma \ref{lemma-1}, for every $i$, there exists $j$ and $l$ such that using 
    \begin{equation*}
          \mathlarger{\mathlarger{\sum}}_{p^k \leq x}  \frac{\mathcal{N}_{p^k}(K_i)\, \log p^k}{n_{K_i}} \geq  \mathlarger{\mathlarger{\sum}}_{p^k \leq x}  \frac{\mathcal{N}_{p^k}(L_j)\, \log p^k}{n_{L_j}} \geq  \mathlarger{\mathlarger{\sum}}_{p^k \leq x}  \frac{\mathcal{N}_{p^k}(K_l)\, \log p^k}{n_{K_l}}.
    \end{equation*}
    Since $\psi_q$ and $\psi_q'$ exist for all prime powers, taking $i$ to infinity, we conclude that for all $x>1$, 
    \begin{equation*}
        \sum_{q\leq x} \psi_q \log q \geq \sum_{q\leq x} \psi_q' \log q \geq \sum_{q\leq x} \psi_q \log q. 
    \end{equation*}
    Hence, $\psi_q = \psi_q'$ for all prime powers $q$.
\end{proof}

For an infinite extension $\rL=\bigcup_i L_i$ over $\Q$, we say that a prime ideal $\rp \in L_j$ almost totally splits in $\rL$ if the number of prime ideals above $\rp$ in $L_t$ is $[L_t:L_j] + o(n_{L_t})$ for $t>j$ as $t\to\infty$. If $\psi_q>0$ for some prime power $q$, then over any tower $\rL = \bigcup_i L_i$, there exists an $L_j$ such that all the primes above $p$ in $L_j$ with norm $q$ almost totally split in $\rL$. More precisely, we have the following.

\begin{proposition}\label{eventual-split}
Let $\rL=\{L_i\}$ be an infinite extension of $\Q$ and $p$ be a fixed rational prime. Suppose $\psi_{p^m} > 0$. Then there exists an integer $N$ such that all prime ideals $\rp \in L_N$ with norm $p^m$ almost totally split in $\rL$. Furthermore, for any $\epsilon>0$, there exists an $M$ such that except at most $\epsilon \, n_{L_M}$, all prime ideals above $p$ in $L_M$ almost totally split in $\rL$. 
    

\end{proposition}
\begin{proof}
To begin with, $\psi_{p^m}$ is well-defined. Therefore, there exists an $N_0$ such that for all $i\geq N_0$, and $j > i$,
    \begin{equation*}
        \frac{\rN_{p^m}(L_j)}{n_{L_j}} - \frac{\rN_{p^m}(L_i)}{n_{L_i}} = o(1),
    \end{equation*}
     as $j\to\infty$. Thus, 
     \begin{equation}\label{spli-eqn}
         \rN_{p^m}(L_j) = [L_j:L_i] \rN_{p^m}(L_i) + o(n_{L_j}).
     \end{equation}
     
\noindent Let $m$ be the smallest positive integer such that $\psi_{p^m} > 0$. Since $\psi_{p^t} =0$ for all $t<m$, the number of prime ideals in $L_i$ with norm $< p^m$ is $o(n_{L_i})$. Thus, almost all the prime ideals in $L_j$ with norm $p^m$ lie above the prime ideals in $L_i$ with norm $p^m$. 
By equation \eqref{spli-eqn}, we deduce that above every $\rp$ in $L_i$ with norm $p^m$, there are $[L_j:L_i] + o(n_{L_j})$ prime ideals in $L_j$, which implies that $\rp$ almost totally splits in $\rL$. \\

\noindent Now suppose $m_1>m$ is the smallest positive integer with $\psi_{p^{m_1}}> 0 $. As $\psi_{p^{m_1}}$ is well-defined, we get that there exists an $N_1$ such that for all $i\geq \max(N_0, N_1)$ and $j > i$, 
\begin{equation*}
  \rN_{p^{m_1}}(L_j) =  [L_j:L_i] \rN_{p^{m_1}}(L_i) + o(n_{L_j})
\end{equation*}
as $j\to \infty$. Note that all prime ideals with norm $p^m$ almost totally split in $\rL$. Also, since $\psi_q = 0$ for all prime ideals with norm $\in (p^m, p^{m_1})$, the number of prime ideals in $L_i$ with norm $\in (p^m, p^{m_1})$ is $o(n_{L_i})$. Therefore, we deduce that $[L_j:L_i] \rN_{p^{m_1}}(L_i) + o(n_{L_j})$  prime ideals in $L_j$ with norm $p^{m_1}$ lie above $\rN_{p^{m_1}}(L_i)$ prime ideals in $L_i$ with norm $p^{m_1}$. Hence, every prime $\rp$ in $L_i$ with norm $p^{m_1}$ almost totally split in $\rL$. Inductively, for any $X>1$, there exists a $N$ such that all prime ideals with norm $p^m \leq X$ in $L_N$ almost totally split in $\rL$ if $\psi_{p^m}>0$.\\

    \noindent
    To prove the second statement, taking $K=\Q$ in Lemma \ref{lemma-1} gives
    $$
        \sum_{q} \psi_q \, \log q \leq \log p,
    $$
    where $q$ runs over all the powers of $p$. Thus, $\sum_{m=1}^{\infty} \psi_{p^m}$ is bounded. Therefore, $\sum_m \psi_{p^m}$ is bounded and for every $\epsilon > 0$, there exists $X>1$ such that 
    \begin{equation*}
        \sum_{p^m>X} \psi_{p^m} < \frac{\epsilon}{2}.
    \end{equation*}
    Using Lemma \ref{lemma-1}, one can find an $M$ uniformly such that for all $i\geq M$,
    \begin{equation*}
    \mathlarger{\mathlarger{\sum}}_{p^k \geq X}  \frac{\mathcal{N}_{p^k}(L_i)\, \log p^k}{n_{L_i}} < \epsilon.
    \end{equation*}
    In other words, the number of prime ideals in $L_i$ with norm $p^m > X$ is at most $\epsilon n_{L_i}$ for $i\geq M$. This proves the proposition.
    \end{proof}

    \medskip

    \begin{lemma}\label{asym-new}
Let $\rL$ be an asymptotically positive extension and $F$ be a number field. Then the compositum $\rL F$ is also an asymptotically positive extension.    
\end{lemma}

\begin{proof}
Let $\rL = \{L_i\}_{i \geq 0}$ be a tower of number fields. Since $\rL$ is asymptotically positive, there exists a prime power $q = p^f$ such that $\psi_q(\rL) > 0$. Let $F$ be a number field of degree $d$, and define the tower $\rL F = \{L_i F\}_{i \geq 0}$.

\medskip

By definition, we have
$$
N_q(L_i) = \psi_q(\rL) \, n_{L_i} + o(n_{L_i}).
$$
It follows that
$$
\sum_{k=1}^d N_{q^k}(L_i F) \geq \psi_q(\rL) \, n_{L_i} + o(n_{L_i}).
$$

\noindent
Dividing both sides by $n_{L_i F}$ and taking the limit as $i \to \infty$, we obtain
$$
\sum_{k=1}^d \psi_{q^k}(\rL F) \geq \frac{\psi_q(\rL)}{d}.
$$

\noindent
Therefore, there exists some $k \in \{1, 2, \ldots, d\}$ such that for $q' = q^k$,
$$
\psi_{q'}(\rL F) \geq \frac{\psi_q(\rL)}{d^2} > 0.
$$
Hence $\rL F$ is asymptotically positive.
\end{proof}

\bigskip

\section{\bf Proof of Theorem \ref{main-theorem} }\label{APE}
\medskip

The outline of the proof of Theorems \ref{main-theorem} (a) is inspired by the argument of Bombieri-Zannier in \cite{BZ}. The new idea is to incorporate relative Mahler's inequality as developed in Section \ref{relative-height}. The proof of Theorem \ref{main-theorem} (b) involves the use of metric property of height and acceleration lemma.

\begin{proof}[Proof of Theorem \ref{main-theorem} (a)] Let $\rL/\Q$ be an infinite extension. For any number field $L\subset \rL$, by Northcott's theorem, there are finitely many $\alpha \in \rL$ with degree $\leq n_L$ such that $h(\alpha) \leq \sum_q \psi_q(\rL)\frac{\log q}{q+1}$, provided this sum converges. \\

For any $X>1$, by Proposition \ref{eventual-split}, there exists a number field $L \subset \rL$ such that for all prime powers $q \leq X$ with $\psi_{q}(\rL)>0$, all prime ideals of norm $q$ in any finite extension of $L$ almost totally split in $\rL$. \\

For any $\alpha \in \mathcal{O}_{\rL}\setminus L$, let $f_L(x)= x^n + a_{n-1}x^{n-1} + \cdots+ a_0$ be its minimal polynomial over $\mathcal{O}_L$. Let $q \leq X$ be such that $\psi_q(\rL)>0$ with $q=p^m$. Let $\rp \subset {\mathcal{O}}_L$ be a prime ideal with norm $q$. Since $\rp$ is almost totally split in $L(\alpha)$, by Lemma \ref{lemma-2}, we can write
\begin{equation*}
    f_L(x) =  (x-\alpha_1) \, (x-\alpha_2) \ldots (x-\alpha_t) \, g(x),
\end{equation*}
where $t= n + o(n)$. If $L_{v}$ is the completion of $L$ with respect to $v$, corresponding to the prime ideal $\rp$, then we have $\alpha_1, \alpha_2, \cdots, \alpha_t \in L_{v}$. Hence, the valuation $v$ can be applied to all $\alpha_i$'s for $i\leq t$.\\

Let $P_{w}$ be the splitting field of $f_L(x)$ over $L_v$ with the unique valuation $w$ extending $v$. Write
    \begin{equation*}
        f_L(x) = (x-\beta_1) (x-\beta_2)\cdots (x-\beta_n).
    \end{equation*}
The discriminant of $f_L$ is given by
\begin{align*}
    D(f_L) & = \prod_{i<j} (\beta_i - \beta_j)^2.
\end{align*}
Therefore,
\begin{align*}
    w(D(f_L)) & \geq  2\sum_{i< j \leq n} w(\beta_j - \beta_i)
\end{align*}
Restricting to $\beta_j$'s which lie in $L_{v}$ gives
\begin{equation*}
    v(D(f_L))  \geq  2\sum_{\substack{i< j \leq n\\\beta_i, \beta_j \in L_v}} v(\beta_j - \beta_i) .
\end{equation*}
The number of $\beta_j \in L_{v}$ is $n+o(n)$. Let $\F_{q} := \rO_L/\rp$ be the residue field at $\rp$. For each $x\in \F_{q}$, let $N_x$ denote the number of roots $\beta_j$ of $f_L$ in $L_v$ which lie in the residue class $x \bmod \rp$. Then, we have the lower bound
\begin{equation*}
    v(D(f_L) \geq \sum_{x\in \F_q} N_x (N_x -1).
\end{equation*}
Since $\sum_{x\in \F_{q}} N_x =n + o(n)$, applying Cauchy-Schwarz inequality, we obtain
\begin{align}\label{valuation_eq}
    v(D(f_L))   &\geq \sum_{x\in \F_q} N_x (N_x -1)  \nonumber\\
                & \geq \frac{n^2}{q} + o(n^2).
\end{align}
Here $n=[L(\alpha):L]$ and error term is $o(n^2)$ as $n\to\infty$.  Taking the sum in equation \eqref{valuation_eq} over all places $v \in M_L$ with norm $\leq X$, we obtain 
\begin{align*}
    \log N_{L/\Q}(D(f_{L})) & = \sum_{v\in M_L} v(D(f_L))\log (\text{Norm}(v))\\
    &\geq \sum_{\substack{v \in M_{L}\\  \text{Norm}(v) \leq X}} v(D(f_L)) \, \log (\text{Norm}(v))\\
    & = \left(\mathlarger{\mathlarger{\sum}}_{q \leq X} \rN_q(L) \, \frac{ \log q \, }{q}\right) \, n^2 + o(n^2),
\end{align*}
where $q$ runs over all prime powers $\leq X$. Using this in the relative Mahler's inequality \eqref{localglobal}, we deduce that
\begin{equation*}
    h(\alpha) \geq \frac{1}{2[L:\Q]} \left(\mathlarger{\mathlarger{\sum}}_{q \leq X} \rN_q(L) \, \frac{ \log q \, }{q}\right)  + o(1).
\end{equation*}
Hence,
\begin{equation*}
    \liminf_{\alpha \in \mathcal{O}_\rL } h(\alpha) \geq \, \frac{1}{2}\, \mathlarger{\mathlarger{\sum}}_{q \leq X}  \, \, \frac{\rN_q(L)}{[L:\Q]} \, \frac{\log q}{q}. 
\end{equation*}
Since $L \subseteq L_i$ for all large enough $i$ and every extension of $L$ has the property that all the ideal of norm $q$ almost totally split in $\rL$, thus
\begin{equation*}
    \liminf_{\alpha \in \mathcal{O}_\rL} h(\alpha) \geq \, \frac{1}{2}\, \mathlarger{\mathlarger{\sum}}_{q \leq X}  \, \, \frac{\rN_q(L_i)}{[L_i:\Q]} \, \frac{\log q}{q} 
\end{equation*}
for all but finitely many $L_i$. Taking the limit as $i \to \infty$ followed by $X\to \infty$, we obtain Theorem \ref{main-theorem} (a).
\end{proof}

\bigskip

Towards the proof of Theorem \ref{main-theorem} (b) we use metric inequalities as in Amoroso–David–Zannier \cite{Amoroso-0}. This consequently proves the property (B) for asymptotically positive extensions. Unfortunately, this method yields a significantly weaker lower bound as compared to Theorem \ref{main-theorem} (a). 

\medskip

Let $L / \Q$ be a number field and $\alpha \in L^{\times}$ not a root of unity. For a non-archimedean place $v$ of $L$ above $p$, denote by $e_v$ and $f_v$ the ramification index and inertia degree of $v$ respectively. Since the residue field corresponding to $v$ has order $q=p^{f_v}$, if $\alpha$ is $v$-integral, we have
$$
|\alpha^{q} - \alpha|_v \leq p^{-1/e_v}.
$$
On the other hand, if $\alpha$ fails to be $v$-integral, then its reciprocal $\alpha^{-1}$ must be $v$-integral, and consequently
$$
|\alpha^{-q} (\alpha^q-\alpha)|_v = |\alpha|_v \ |\alpha^{-q} - \alpha^{-1}|_v \leq p^{-1/e_v} \ |\alpha|_v \implies |\alpha^q-\alpha|_v \leq p^{-1/e_v} \ |\alpha|_v^{q+1} 
$$
\medskip

\noindent
Therefore, for non-archimedean places $v \mid p$ of norm $q$, we get

\begin{align}\label{metric-1}
\left|\alpha^{q} - \alpha\right|_v &\leq c(v) \max \left(1,|\alpha|_v\right)^{{q} + 1},
\end{align}

\noindent
where $c(v)=p^{-1/e_v}$. For all other places $w\in M_L$,
$$
\left|\alpha^{q} - \alpha\right|_w \leq c(w) \max \left(1,|\alpha|_w\right)^{{q} + 1},$$
where
$$
c(w)= \begin{cases}1, & \text { if } w \nmid \infty, \ N(w) \neq q; \\ 2, & \text { if } w \mid \infty.\end{cases}
$$
Since $\alpha$ is not a root of unity, $\alpha^{q} - \alpha \neq 0$. Hence, the product formula gives
$$
\begin{aligned}
0 & =\sum_{w\in M_L} \frac{\left[L_w: \mathbb{Q}_w\right]}{[L: \mathbb{Q}]} \log \left|\alpha^{q} - \alpha\right|_w \\ 
& \leq \sum_{w\in M_L} \frac{\left[L_w: \mathbb{Q}_w\right]}{[L: \mathbb{Q}]}\,\left(\log c(w)+ (q + 1 ) \log \max \left\{1,|\alpha|_w\right\}\right) \\
& =\left(\sum_{w \mid \infty} \frac{\left[L_w: \mathbb{Q}_w\right]}{[L: \mathbb{Q}]}\right) \log 2-\left(\sum_{w \mid p, N(w)=q} \frac{\left[L_w: \mathbb{Q}_w\right]}{[L: \mathbb{Q}]}\right)  \frac{\log p}{e_w} + (q + 1) h(\alpha) \\
& =\log 2- \frac{\rN_q(L)}{[L : \Q]} \log q + (q+1) h(\alpha).
\end{aligned}
$$

\medskip

\noindent
So, we immediately deduce that

$$
h(\alpha) \geq \frac{1}{q+1}\left(\frac{\rN_q(L)}{[L : \Q]}  \log q - \log 2 \right).
$$

\medskip
\noindent
Suppose $\rL=\bigcup_i L_i$. Taking liminf, we get

\begin{equation}\label{ATS-Metric}
\liminf_{\alpha\in \rL\setminus \mu_{\infty}} h(\alpha) \geq \frac{1}{q+1}\left(\psi_q(\rL) \log q - \log 2 \right).
\end{equation}

\medskip

\noindent
This bound is meaningful only if 
$$
  {\psi_q}(\rL) > \frac{2}{\log q},
$$
which may not hold, as $\psi_q$ can take arbitrarily small positive values. To overcome this limitation, we use the Lemma \ref{acc-lemma} with $\rho = 1/e_v$. Proceeding similarly to \eqref{metric-1}, we obtain
\begin{equation}\label{metric-2}
\left|\alpha^{p^{f_v + \lambda}} - \alpha^{p^\lambda}\right|_v  \leq c'(v) \max \left(1,|\alpha|_v\right)^{p^{(f_v +\lambda)} + p^\lambda},
\end{equation}
where $c'(v) = p^{-s_{p, \rho}(\lambda)}$. 

\begin{proof}[Proof of Theorem \ref{main-theorem} (b)]
Suppose $\rL/\Q$ is an asymptotically positive extension with $\psi_q(\rL)>0$ for $q=p^f$. Write $\rL=\bigcup_{i=1}^{\infty} L_i$ as a tower of number fields. \\ 

\noindent
In order to apply \eqref{metric-2}, note that $s_{p, \rho}(\lambda)$ depends on both $\lambda$ and the place $v$. For our application, we first show that we can choose $\lambda$ depending only on $p$ such that the $s_{p, \rho}(\lambda)$ is bounded below uniformly, for a positive proportion of places above $p$. Indeed
$$
\rN_q(L_i) \leq  \sum_{\substack{v | p, \\ N(v) = q}} e_v \leq [L_i : \Q]
$$
and hence,
$$
\frac{1}{\rN_{q}(L_i)}\sum_{\substack{v | p, \\ N(v)=q}} e_v \leq \frac{[L_i:\Q]}{\rN_q(L_i)}.
$$
For sufficiently large $i$, let $\epsilon_i$ be such that the number of $v\mid p$, with $N(v)= q$ and $e_v \leq \frac{2}{\psi_q(\rL)}$ is given by $(1-\epsilon_i)\psi_q(\rL) [L_i:\Q]$. Clearly, $\epsilon_i \to 0$ as $i\to\infty$. For all such places with $e_v \leq \frac{2}{\psi_q(\rL)}$, the value of $\rho=\frac{1}{e_{v}}$ is uniformly bounded below. Therefore, for all such places, $\lambda$ can be chosen so that $s_{p, \rho} (\lambda)$ is bounded below uniformly. \\

\noindent Using metric estimates for
$\left| \alpha^{{p^\lambda + q}} - \alpha^{p^\lambda} \right|_v$ obtained in \eqref{metric-2}, we deduce that
\begin{align*}
p^{\lambda}(q + 1) \ h(\alpha) &\geq \left(\sum_{\substack{v \mid p, \ N(v)=q \\ e_v \leq \frac{2}{\psi_q(\rL)}}} \frac{\left[(L_i)_v: \mathbb{Q}_p\right]}{[L_i: \mathbb{Q}]}\right)  s_{p, \rho}(\lambda)\log p - \log 2 \\
& \geq s_{p, \rho} (\lambda) (1-\epsilon_i) \, \psi_q(\rL)\, \log q - \log 2,
\end{align*}
Since $\epsilon_i \to 0$ as $i\to \infty$, we conclude that
$$
\liminf_{\alpha \in \rL\setminus \mu^{\infty}} \ \ h(\alpha) \geq \frac{1}{p^{\lambda}(q + 1)} \left( s_{p, \rho}(\lambda) \ \psi_q(\rL) \log q - \log 2 \right).
$$

\medskip
\noindent
Now choosing $\lambda$ large enough so that $s_{p, \rho}(\lambda) > 2/\psi_q(\rL)$, we deduce that
$$
\liminf_{\alpha \in \rL\setminus \mu^{\infty}} \ \ h(\alpha) > \frac{\log q}{p^{\lambda}(q + 1)}.
$$
This proves the theorem.
\end{proof}

\noindent


\bigskip

\section{\bf Proof of Proposition \ref{almost-split} and Theorem \ref{almost-unramified}}\label{ATS-AU}

The central ideas in the proof of Proposition \ref{almost-split} and Theorem \ref{almost-unramified} are similar to those used in Theorem \ref{main-theorem} (a), though a more careful analysis is required. For almost totally split extensions, unlike in the proof of Theorem \ref{main-theorem} (a), we do not restrict ourselves to algebraic integers, as over the base field $\Q$ the relative Mahler measure coincides with the usual Mahler measure.

\begin{proof}[Proof of Proposition \ref{almost-split}]
    Let $\rL$ be an infinite extension and $S$ be a set of rational primes such that $\rL$ is almost totally split at each $p \in S$.\\
    
    For any $\alpha \in \rL$, let $L=\Q(\alpha)$ and $f(x) \in \Z[x]$ be its minimal polynomial of degree $n$. By Lemma \ref{lemma-2}, the minimal polynomial $f$ factorizes over $\Q_p$ as 
    \begin{equation*}
        f(x)= a_n (x-\alpha_1)(x-\alpha_2) \cdots (x-\alpha_t) g(x),
    \end{equation*}
    where $t = n + o(n)$ and $\deg \, (g(x)) = o(n)$ as $n$ goes to infinity. Thus the $p$-adic valuation $v$ on $\Q_p$ can be applied to $\alpha_1, \cdots, \alpha_t$. Let $L_{w}$ be the splitting field of $f(x)$ over $\Q_p$ with valuation $w$. Write
    \begin{equation*}
        f(x) = a_n (x-\beta_1) (x-\beta_2)\cdots (x-\beta_n), 
    \end{equation*}
    where $\beta_i\in L_{w}$ satisfying
    \begin{equation*}
        w(\beta_1)\geq \cdots \geq w(\beta_r)\geq 0 > w(\beta_{r+1})\geq \cdots \geq w(\beta_n).
    \end{equation*}
    Note that except for $o(n)$  many $\beta_i$'s, the rest are same as $\alpha_i$'s up to ordering. Consider the discriminant 
\begin{align*}
    D(f) & = a_n^{2n-2}\prod_{i<j} (\beta_i - \beta_j)^2.
\end{align*}
The contribution in the product where at least one $w(\beta_j)<0$ can be bounded by
\begin{equation*}
    w\left(\prod_{j=r+1}^{n} \prod_{i=1}^{j-1} (\beta_j - \beta_i)\right) \geq \mathlarger{\mathlarger{\sum}}_{j=r+1}^n (j-1) \, w(\beta_j).
\end{equation*}
Hence, we obtain
\begin{align}\label{eqn-1}
    w(D(f)) & \geq (2n-2)w(a_n) + 2 \, \mathlarger{\mathlarger{\sum}}_{0<i<j\leq r} w(\beta_j-\beta_i) \, +  \, 2 \, \mathlarger{\mathlarger{\sum}}_{j=r+1}^{n} (j-1) \, w(\beta_j) \nonumber\\
    & \geq 2 \, \mathlarger{\mathlarger{\sum}}_{i< j \leq r} w(\beta_j - \beta_i) \, - \,  2 \, \mathlarger{\mathlarger{\sum}}_{j=r+1}^{n} (n-j) \, w(\beta_j).
\end{align}
Recall that $f(x)$ has $n+ o(n)$ roots in $\Q_p$. Let $e_{w}$ be the ramification index of $L_{w}/\Q_p$. For $\beta \in \Q_p$, we have $w(\beta) = v(\beta) \, e_{w}$. Since all terms in the RHS of \eqref{eqn-1} are positive, we can write
\begin{align}\label{eqn-11}
     v(D(f))  = \frac{w(D(f))}{e_{w}} & \geq \frac{2}{e_{w}} \, \mathlarger{\mathlarger{\sum}}_{i< j \leq r} w(\beta_j - \beta_i) \, - \, \frac{2}{e_{w}} \, \mathlarger{\mathlarger{\sum}}_{j=r+1}^{n} (n-j) \, w(\beta_j) \nonumber\\ 
     &\geq 2 \, \mathlarger{\mathlarger{\sum}}_{\substack{i< j \leq r\\\beta_i, \beta_j \in \Q_p}} v(\beta_j - \beta_i) \, - \,  2 \, \mathlarger{\mathlarger{\sum}}_{\substack{r<j\leq n \\ \beta_j \in \Q_p}} (n-j) \, v(\beta_j).
\end{align}
Let $N_x$ denote the number of $v$-integral roots $\beta_j$ of the polynomial $f$ in $\Q_p$, which lie in the residue class $x \bmod p$. If $\beta_i,\beta_j \in \Q_p$ lie in the same residue class modulo $p$, then $v(\beta_i-\beta_j) \geq 1$. Hence, 
\begin{equation*}
    \mathlarger{\mathlarger{\sum}}_{\substack{i< j \leq r\\ \beta_i, \beta_j \in \Q_p}} v(\beta_j - \beta_i) \geq \mathlarger{\mathlarger{\sum}}_{x\in \F_p} \frac{N_x (N_x -1)}{2}.
\end{equation*}
Therefore, by \eqref{eqn-11}, we have the lower bound
\begin{equation*}
    v(D(f)) \geq \mathlarger{\mathlarger{\sum}}_{x\in \F_p} N_x (N_x -1) +  (n-r)(n-r-1).
\end{equation*}
Since the number of roots of $f(x)$ in $\Q_p$ is $n-r +\sum_{x\in \F_p} N_x = n + o(n) $, applying Cauchy-Schwarz inequality, we obtain
\begin{align*}
    v(D(f))   
                & \geq \frac{n^2 + o(n^2)}{p+1}.
\end{align*}
Summing over all primes $p \in S$ with $p \leq X$, we obtain
\begin{equation*}
    \log D(f) \geq  \left(\sum_{p \in S, \, p \leq X } \, \frac{\log p}{p+1} \,\right) \left(n^2 + o(n^2) \right).
\end{equation*}
Now, applying Theorem \ref{mahler}, we deduce

\begin{align*}
    h(\alpha) = \frac{\log M(\alpha)}{n} & \geq \frac{(n-1) \log M(f)}{n^2}\geq \frac{\log D(f)}{2n^2} - \frac{\log n}{n}\\ 
    & \geq \frac{1}{2} \left(\sum_{p \in S, \, p \leq X } \, \frac{\log p}{p+1} \,\right) \left(1 + o(1) \right) - \frac{\log n}{n}.
\end{align*}
As $n\to \infty$, we have
\begin{equation*}
    \liminf_{\alpha \in \rL} h(\alpha) \geq \frac{1}{2}\sum_{p \in S, \, p \leq X } \frac{\log p}{p+1}.
\end{equation*}
Finally, taking $X \to \infty$ proves Proposition \ref{almost-split}.
     
\end{proof}

\medskip

\begin{remark*}
Let $\rL$ be an extension which is almost totally split at a rational prime $p$. Then, for any $\alpha \in \rL$ that is not a root of unity, it follows from \eqref{ATS-Metric} that
$$
 h(\alpha) \geq \frac{\log (p/2)}{p+1} + o(1),
$$
where $o(1)$ tends to zero as $\deg(\alpha) \to \infty$. In the totally p-adic case, this error term vanishes, yielding Pottmeyer's bound in \eqref{Pottmeyer} exactly.
\end{remark*}

\bigskip

We now establish lower bounds for logarithmic Weil heights of elements in an almost unramified asymptotically positive extension. Recall that an infinite extension $\rL/\Q$ is said be almost unramified asymptotically positive at a prime $p$ if
$$
\mathlarger{\mathlarger{\sum}}_{{f\geq 1}} \, f \, \psi_{p^f} = 1.
$$
In other words, if $\rL=\bigcup_i L_i$,
$$
   \mathlarger{\mathlarger{\sum}}_{f=1}^{\infty} \, f \, \rN_{p^f}(L_i) = n_{L_i} + o(n_{L_i}).
$$

\begin{proof}[Proof of Theorem \ref{almost-unramified}]

We first note that for $L\subset \rL$, the number of prime ideals in $L$ with ramification index $\geq 2$ is $o(n_L)$. Indeed, suppose $p\rO_L = \rp_1^{e_1}\cdots \rp_g^{e_g}$. Because $\rL$ is almost unramified asymptotically positive, we have 
\begin{equation*}
    n_L = \mathlarger{\mathlarger{\sum}}_{f=1}^{\infty} \, f\,\rN_{p^f}(L) + o(n_L) = \sum_{j\leq g} f_j + o(n_L),
\end{equation*}
 where $f_j$ is the residue class index of $\rp_j$. As  $n_L = \sum_j e_j f_j$, we deduce that $e_j \geq 2$ for at most $o(n_L)$ prime ideals above $p$.\\

\noindent
Since $\rL$ is almost unramified asymptotically positive at $p$, for a given $\epsilon > 0$, there exists an integer $N$ such that
$$
  1-\epsilon \leq \mathlarger{\mathlarger{\sum}}_{f=1}^N f \, \psi_{p^f} \leq  1.
$$

\noindent
Thus for $\alpha \in \rL$ and $L = \Q(\alpha)$ with $[L:\Q]$ large enough,

$$
    (1-\epsilon) \,  n_{L} \leq \sum_{f=1}^{N} f\, \rN_{p^f}(L) \leq n_{L}.
$$
\\

Let $
f(x)=a_{n_L} x^{n_L} + a_{n_{L-1}} x^{n_{L-1}} + \cdots + a_0
$ be the minimal polynomial of $\alpha$ over $\Q$ and $\widetilde{L}_{\omega}$ be the splitting field of $f(x)$ over $\Q_p$. Suppose the conjugates of $\alpha$ given by $\alpha_1, \alpha_2, \cdots , \alpha_{n_L}$ be ordered such that

\begin{equation*}
\omega({\alpha_1}) \geq  \omega({\alpha_2}) \geq \ldots \geq \omega({\alpha_r}) \geq 0 > \omega({\alpha_{r+1}}) \geq \ldots \geq \omega({\alpha_{n_L}}).
\end{equation*}

\medskip

\noindent
Let $S_1, S_2, \ldots, S_N$ denote the set of unramified prime ideals in $L$ above $p$, with norm $p, p^2, \ldots, p^N$ respectively. Thus, for each place $\nu_k$ in $S_k$, $L_{\nu_k}$ is an unramified extension of $\Q_p$ of degree $k$. By the uniqueness of the unramified extensions of $\Q_p$ of a given degree and the fact that $\widetilde{L}_{\omega}$ is the splitting field of $f(x)$, it follows that $L_{\nu_k} \subseteq \widetilde{L}_{\omega}$, for $1 \leq k \leq N$. Moreover, each $L_{\nu_k}$ has exaclty $|S_k|$ many roots of $f(x)$ that generate $L_{\nu_k}$. Hence, the roots of $f(x) \in \widetilde{L}_{\omega}$ that generate $L_{\nu_k}$ are in one to one correspondence with $S_k$.

\medskip

For $1 \leq k \leq N$, let $\alpha_{1,\,k}, \, \alpha_{2,\,k}, \ldots, \, \alpha_{|S_k|,\,k}$ be the roots of $f(x)$ that generate $L_{\nu_k}$. Grouping these roots together, we can write
   \begin{align*}
       f(x) = \,\, a_{n_L} \, & (x-\alpha_{1,\,1})\,(x-\alpha_{2,\,1})\,\ldots(x-\alpha_{|S_1|,\,1})\\
       & (x-\alpha_{1,\,2})\,(x-\alpha_{2,\,2})\,\ldots(x-\alpha_{|S_2|,\,2})\\
       & \hspace{0.8 cm} \vdots \hspace{1.6 cm} \vdots  \hspace{2.2 cm} \vdots \\
       & (x-\alpha_{1,\,N})\,(x-\alpha_{2,\,N})\,\ldots(x-\alpha_{|S_N|,\,N}) \,\,g(x)
   \end{align*}
where $g(x)$ has degree $o(n_L)$. The roots are ordered such that for a fixed $k$, 

$$
\omega(\alpha_{1,\,k}) \geq \omega(\alpha_{2,\,k}) \geq \ldots \geq \omega(\alpha_{r_k,\,k}) \geq 0 >  \omega(\alpha_{r_k + 1,\,k}) \geq \omega(\alpha_{|S_k|,\,k}).
$$

\medskip

\noindent
Let $D(f)$ be the discriminant of $f(x)$. Then, we have
\begin{align}\label{Rearrange}
\omega(D(f)) \,\,\, \geq& \,\,\,\, 2 \sum_{1 \leq i < j \leq r} \omega(\alpha_i - \alpha_j) - 2 \sum_{j=r+1}^{n_L} (n_L - j) \, \omega(\alpha_{j})\\
\geq &\,\,\,\, 2 \sum_{k=1}^N \left( \sum_{i < j \leq r_k} \omega(\alpha_{i,k} - \alpha_{j,k}) - \sum_{j=r_k+1}^{|S_k|} (n_L - t_{j,k}) \,\, \omega (\alpha_{j,k}) \right),\nonumber
\end{align}

\noindent
where $t_{j,k}$ is the integer such that $\alpha_{t_{j,k}}=\alpha_{j,k}$. The last inequality follows from the fact that each term on the RHS of \eqref{Rearrange} is non-negative and we are considering only those roots that generate $L_{\nu_k}$, for $1 \leq k \leq N$.\\ 

Let $\nu$ denote the usual $p-$adic valuation on $\Q_p$. Since $\omega$ and $\nu_k$ are the unique extensions of $\nu$ to $\widetilde{L}_{\omega}$ and $L_{\nu_k}$ respectively, we have 

\begin{align*}
 \omega(D(f)) & = {\nu(D(f))} \, {e(\widetilde{L}_{\omega}|\Q_p)}, \\ 
 \omega(\alpha_i - \alpha_j)& ={\nu_k(\alpha_{i,l} - \alpha_{j,k})} \, {e(\widetilde{L}_{\omega} | L_{\nu_k})} \,\,\,\,\, \text{and} \\ 
  \omega(\alpha_j) & = {\nu_k(\alpha_{i})} \, {e(\widetilde{L}_{\omega} | L_{\nu_k})}.
\end{align*}

\medskip

\noindent
Furthermore, $L_{\nu_k} \slash \Q_p$ is unramified and hence, $e(\widetilde{L}_{\omega} | L_{\nu_k}) = e(\widetilde{L}_{\omega} | \Q_p)$. Thus, 
$$
\nu(D(f)) \geq \,\,\,\, 2 \sum_{k=1}^N \left( \sum_{i < j \leq r_k} \nu_k (\alpha_{i,k} - \alpha_{j,k}) - \sum_{j=r_k+1}^{|S_k|} (n_L - t_{j,k}) \,\, \nu_k (\alpha_{j,k}) \right).
$$

\medskip

The roots $\alpha_{1,\,k}, \, \alpha_{2,\,k}, \ldots, \, \alpha_{r_k,\,k}$ lie in $L_{\nu_k}$. Since the residue field of $L_{\nu_k}$ is $\F_{p^k}$, for each $x \in \F_{p^k}$, let $N_{x,k}$ denote the number of $\nu$-integral ${\alpha_{i,k}}$'s that lie in the residue class $x \mod p^k$. Also, note that the numbers $(n_L - t_{j, k})$ appearing in \eqref{Rearrange} are distinct elements from the set $\{n_L-j \, : \, r_k + 1 \leq j \leq n_L \}$ and hence, form a set of distinct non-negative integers. Thus, we have the trivial bound
\begin{align*}
 - \sum_{j=r_k+1}^{|S_k|} (n_L - t_{j,k}) \,\, \nu_k (\alpha_{j,k}) \,\, \geq& \,\,\,\, \sum_{j=0}^{|S_k| - r_k-1} \, j \, 
= \,\, \frac{\left(  |S_k| - r_k-1 \right) \left( |S_k| - r_k \right)}{2}.
\end{align*}

\medskip

\noindent
Therefore
\begin{align*}
\nu(D(f)) \,\,\, \geq& \,\,\,\, \mathlarger{\mathlarger{\sum}}_{k=1}^N \, \, \left( 
\sum_{x \in \F_{p^k}} N_{x,k} \left(N_{x,k}-1 \right)  + \left(  |S_k| - r_k-1 \right) \left( |S_k| - r_k \right) \right) \\
 \geq& \,\,\,\, \mathlarger{\mathlarger{\sum}}_{k=1}^N \, \, \left( 
\sum_{x \in \F_{p^k}} {N_{x,k}}^2 + \left(|S_k|-r_k \right)^2 \right) + O(n_L).
\end{align*}

\medskip

\noindent
Since $ \left(|S_k|-r_k \right) + \sum_{x \in \F_{p^k}} N_{x,k} = |S_k| $, by Cauchy-Schwarz inequality, we obtain

\begin{align*}
\nu_k(D(f)) \,\,\, \geq& \,\,\,\, \mathlarger{\mathlarger{\sum}}_{k=1}^N \, \,  \left( \frac{|S_k|^2}{p^k+1} \right) + O(n_L).
\end{align*}

\medskip

\noindent
As $ |S_k| = k \, \rN_{p^k}(L) + o(n_L)$, we deduce that
\begin{align*}
\nu (D(f)) \,\,\, \geq \,\,\,\, \mathlarger{\mathlarger{\sum}}_{k=1}^N \,\, \left( \frac{k^2 {\rN_{p^k}(L)}^2}{p^k+1} \right) + O(n_L) \,\,\,\, \geq& \,\,\,\, {n_L}^2  \,\, \mathlarger{\mathlarger{\sum}}_{k=1}^N  \, \, \frac{k^2}{p^k+1} \left( \frac{\rN_{p^k}(L)}{n_L} \right)^2 \, + \, o({n_L}^2).
\end{align*}

\noindent
Therefore,
$$
\log D(f) \,\, \geq \,\, {n_L}^2  \,\, \mathlarger{\mathlarger{\sum}}_{k=1}^N  \, \, \frac{k^2}{p^k+1} \left( \frac{\rN_{p^k}(L)}{n_L} \right)^2 \log q \, + \, o({n_L}^2).
$$

\medskip

\noindent
Now, using Theorem \ref{mahler} and taking $N \rightarrow \infty$, we conclude that
\begin{equation*}
    \liminf_{\alpha \in \rL} \, h(\alpha) \geq \mathlarger{\mathlarger{\sum}}_{{k\geq 1}} \, k^3 \, \left(\psi_{p^k}\right)^2 \frac{\log p}{2\left(p^k+1 \right)}.
\end{equation*}
\end{proof}

\bigskip

\section{\bf Canonical height on elliptic curves}\label{canonical_height}

\bigskip

Let $E$ be an elliptic curve defined over a number field $K$. Recall that the canonical height of a point $P \in E(\overline{K})$ is given by 
\begin{equation*}
    \widehat{h}(P) := \lim_{n \rightarrow \infty} \frac{h([2^n] \cdot P)}{4^n},
\end{equation*}
where $h$ is the Weil height. The canonical height $\widehat{h}$ is a quadratic form on $E(\overline{K})$ and satisfies the following properties (see \cite{Neron} or \cite[Chapter IV, \textsection 3 ]{Lang1}) : {

\medskip

\begin{enumerate}
    \item[(a)] $\widehat{h}(P) \geq 0$ for all $P \in E(\overline{K})$, and $\widehat{h}(P)=0$ iff $P$ is a torsion point.\\
    \item[(b)] $\widehat{h}([n] \cdot P) = n^2 \, \widehat{h}(P)$ for all $n \in \Z$.\\
    \item[(c)] $\widehat{h}(P_1+P_2)+\widehat{h}(P_1-P_2) = 2\widehat{h}(P_1)+2\widehat{h}(P_2)$.
\end{enumerate}
\medskip

\noindent Lehmer's conjecture for elliptic curves, as formulated in \cite{Anderson and Masser} can be stated as follows.
\begin{conjecture}
    Let $E$ be an elliptic curve over a number field $K$. For any $P\in E(\overline{K})$, let $K(P)$ denote the field obtained by adjoining the coordinates of $P$ to $K$. Then, there exists a constant $C(E,K)$ depending only on $E$ and $K$ such that 
    $$
\widehat{h}(P) > \frac{C(E,K)}{[K(P):K]},
    $$
    for all non-torsion points $P \in E(\overline{K})$.
\end{conjecture}

This remains open but weaker lower bounds have been obtained in special cases. For a fixed number field $K$ and a non-torsion point $P\in E(\overline{K})$, denote by $D= [K(P):K]$. Some classical results are listed below.\\

\begin{center}

\begin{tabular}{l|l|l}

$\widehat{h}(P) \geqslant$ & Restriction on $E$ & Reference \\
\hline
$c D^{-10}(\log D)^{-6}$ & none & Anderson-Masser (1980) \cite{Anderson and Masser} \\
$c D^{-1}\left(\frac{\log \log (D)}{\log (D)}\right)^3$ & $\mathrm{CM}$ & Laurent (1983) \cite{Laurent} \\

$c D^{-3}(\log D)^{-2}$ & none & Masser (1989) \cite{Masser} \\

$c D^{-2}(\log D)^{-2}$ & $j$ non-integral & Hindry-Silverman (1990) \cite{Hindry} \\
\end{tabular}
\smallskip 
\captionof{table}{Bounds on $\widehat{h}$ for $E(\overline{K})$}
\end{center}

\medskip

Analogous to the case of algebraic numbers, for an elliptic curve $E$ over $K$ and $S \subseteq \overline{K}$, $E(S)$ is said to have \textit{Bogomolov} property (B) if there exists a constant $C(E,K) > 0$, such that $
\widehat{h}(P) > C(E,K)$, for all non-torsion points $P \in E(S)$.
\medskip

Let $\rL/\Q$ be an infinite extension satisfying property (B) with respect to the Weil height. Our objective is to understand if $E(\rL)$ satisfies property (B) with respect to the canonical height. The first result of this kind is due to S. Zhang \cite{Zhang}, who showed that for an elliptic curve $E/\Q$, the set of totally real points $E(\Q^{tr})$ satisfies property (B). This can be thought of as the analog of Schinzel's result \cite{Schinzel} for $\Q^{tr}$. In 2002, M. Baker proved property (B) for $E(K^{ab})$ when $E$ is CM or has non-integral $j$-invariant. This is the analog of the theorem of Amoroso-Zannier \cite{Amoroso II} for $K^{ab}$. Moreover, Baker's bound only depends on the $j-$invariant of $E$ and is effectively computable. Later, Silverman \cite{Silverman} proved property (B) for $E(K^{ab})$ when $E$ is non-CM with the constant being ineffective. Some known bounds on $\widehat{h}$ over infinite extensions are summarized in the table below.\\

\medskip

\begin{center}

\begin{tabular}{c|c|c}
${\widehat{h}(P)} \geqslant$ & Restriction on $E$ & Reference \\
\hline
$c D^{-2}$ & none & Silverman (1981) \cite{SilvermanII} \\

$c D^{-1}(\log D)^{-2}$ & none & Masser (1989) \cite{Masser} \\

$c D^{-2 / 3}$ & $j$ non-integral & Hindry-Silverman (1990) \cite{Hindry} \\

$c$ & $j$ non-integral or CM & Baker (2003) \cite{Baker} \\

$c$ & none & Silverman (2003) \cite{Silverman} \\

\end{tabular}
\smallskip 
\captionof{table}{Bounds on $\widehat{h}$ for $E(K^{ab})$}
\end{center}

\medskip

Recall that a subfield $\rL$ of $\overline{\Q}$ is said to be totally $p$-adic of type $(e,f)$ if for any place $v$ in $\rL$ above $p$,  $\rL_{v}/\Q_p$ is a finite extension with ramification index and residue class degree bounded above by $e$ and $f$ respectively. Let $E/K$ be an elliptic curve with semistable reduction at all places above $p \neq 2$. In 2005, M. Baker and C. Petsche \cite{Baker-2} showed that if $\rL / K$ is a totally $p$-adic field of type $(e, f)$ for a prime $p$, then $E(\rL)$ has property (B). This can be regarded as the elliptic analog of Bombieri-Zannier's result Theorem \ref{B-Z}. In this case, they also gave an effective upper bound for the number of torsion points in $E(\rL)$. Let $E/\Q$ be an elliptic curve and $\Q(E_{tor})$ denote the field generated by all torsion points of  $E$ in $\overline{\Q}$. In 2013, Habegger \cite{Habegger} showed that both $\Q(E_{tor})$ and $E(\Q(E_{tor}))$ have property (B) with respect to the Weil and the canonical height respectively. Note that the Kronecker-Weber theorem implies that $\Q^{ab}$ is generated by all the roots of unity in $\overline{\Q}$. Therefore, the field $\Q(E_{tor})$ can be thought of as the elliptic analog of $\Q^{ab}$.\\

Let $\rL/\Q$ be an asymptotically positive infinite extension. By Theorem \ref{main-theorem}, we know that $\rL$ satisfies property (B). Our goal is to prove that $E(\rL)$ satisfies property (B) for any elliptic curve $E$. Towards this, we introduce  some notation.
\medskip

Consider an elliptic curve $E/K$ and a tower $\rL = \{ L_i \}_{i \geq 0}$ containing $K$. For any rational prime power $q=p^k$, denote by $G_q(L_i)$ and $B_q(L_i)$ the number of prime ideals in $L_i$ of norm $q$ with good reduction and split multiplicative reduction at $E/L_i$ respectively. Define
\begin{equation*}
    \xi_q ( \rL ) : = \lim_{i \to \infty}\frac{G_q(L_i)}{[L_i : \Q]} \,\,\,\,\,
    \text{ and } \,\,\,\,
    \chi_q ( \rL ) : = \lim_{i \to \infty}\frac{B_q(L_i)}{[L_i : \Q]}.
\end{equation*}

Since every prime of bad reduction attains split multiplicative reduction after a finite base change, there exists a number field $L/K$ such that $E/L$ has semistable reduction i.e. either good or split multiplicative reduction at every non-archimedean place in $L$. Suppose $\rL=\bigcup_{i=1}^{\infty} L_i$ be an asymptotically positive extension and $L\subset \rL$, then clearly
\begin{equation*}
    \psi_q(\rL) = \xi_q (\rL)+ \chi_q(\rL).
\end{equation*}
On the other hand, if $L\not\subset \rL$, then note that $\rL\otimes_K L = \bigcup_{i=1}^{\infty} L_i\otimes_K L$ is also an asymptotically positive extension. Since $E/L$ has semistable reduction over all non-archimedean primes, we have, for all primes powers $q$,
\begin{equation*}
    \psi_q(\rL\otimes_K L) = \xi_q (\rL\otimes_K L)+ \chi_q(\rL\otimes_K L).
\end{equation*}
\medskip



\noindent
\textbf{Example.} Note that, in the above paragraph, the desired reduction property of the elliptic curve $E$ holds for $\rL':= \rL\otimes_K L$, but not necessarily for $\rL$. For instance, let $E/\Q$ be the elliptic curve given by $y^2=x^3+7 x$ and let $\rL:=\Q^{t_7}$, the maximal totally $7$-adic extension of $\Q$. Then, $\rL$ is asymptotically positive with $\psi_7(\rL)=1$ and for any $L_i\subset \rL$, $E/L_i$ has additive reduction at all the places of $L_i$ above $7$. However, in $L=\Q(\sqrt{-1})$, the prime $7$ remains inert and $E$ has split multiplicative reduction over the unique prime above $7$ in $L$. Now, considering the infinite extension $\rL' = \rL\otimes_{\Q} L$, we note that $\rL'$ is asymptotically positive with $\psi_{49}(\rL')>0$ and $E$ has semistable reduction over all the primes above $7$ in $\rL'$.\\


Our main theorem is as follows.

\begin{theorem}\label{Elliptic Case I}
Let $E/ K$ be an elliptic curve and $\rL=\bigcup_{i}L_i$ be an infinite extension over $K$. Then,
\begin{equation*}
\liminf_{P\in E(\rL)}\,\,\widehat{h}(P) \geq \frac{1}{48} 
 \mathlarger{\mathlarger{\sum}}_{\substack{q}} \left( 6 \, \left( \frac{q+1}{q+ 1 + 2\sqrt{q}} \right) \,  \xi_q(\rL) + (q+1) \, c_E \, \chi_q(\rL)  \right) \, \frac{\log q}{q+1},
\end{equation*}
where the sum runs over all prime powers $q$ and $c_E$ is a positive constant depending on $E$ and $K$. 
\end{theorem}
\medskip

This answers a question raised by P. Fili and Z. Miner \cite{Fili-1} about obtaining a quantitative lower bound for canonical height of points on elliptic curves over totally $p$-adic extensions. This question appears in the pre-print version of the paper (see page 4 of \url{https://math.okstate.edu/people/fili/equi-splitting.pdf}).\\

\medskip

As an immediate corollary to Theorem \ref{Elliptic Case I} and Lemma \ref{asym-new}, we deduce Theorem \ref{elliptic-analogue} as stated in the introduction.
\begin{corollary}
    Let $E/K$ be an elliptic curve. Let $\rL/\Q$ be an asymptotically positive extension containing $K$. Then $E(\rL)$ satisfies property (B).
\end{corollary}
\begin{proof}
    Let $L/K$ be a finite extension such that $E/L$ has good reduction or split multiplicative reduction over all non-archimedean places in $L$. Since $\rL/\mathbb{Q}$ is asymptotically positive, Lemma~\ref{asym-new} implies that $\rL \otimes_K L$ is also an asymptotically positive extension and
    \begin{equation*}
    \psi_q(\rL\otimes_K L) = \xi_q (\rL\otimes_K L)+ \chi_q(\rL\otimes_K L) > 0
\end{equation*}
    for some prime power $q$. Now, Theorem \ref{Elliptic Case I}, implies that $E(\rL\otimes_K L)$ satisfies property (B) and hence, $E(\rL)$ satisfies property (B).
\end{proof}
\medskip

Let $E/\Q$ be an elliptic curve with good reduction at $p$. For $q=p^m$ and any infinite extension $\rL/\Q$, by definition $\psi_q(\rL) = \xi_q(\rL)$. Hence, we have the following corollary. 
\begin{corollary}\label{Elliptic_Case_II}
Let $E/ \Q$ be an elliptic curve and $\rL$ be an infinite extension of $\Q$. Then,
\begin{equation*}
\liminf_{P\in E(\rL)}\,\widehat{h}(P) \geq \frac{1}{8}  \sideset{}{'}\sum_{\substack{q}} \psi_q(\rL) \, \frac{\log q}{q+ 1 + 2\sqrt{q}},
\end{equation*}
where $\sideset{}{'}\sum$ runs over all prime powers $q=p^m$ such that $E$ has good reduction at $p$. In particular, if $\psi_{p^m}>0$ and $E$ has good reduction at $p$, then $E(\rL)$ satisfies property (B).
\end{corollary}
\medskip

The condition in the corollary above can be relaxed from primes of good reduction to primes of potentially good reduction for $E$. This is because if $E/\Q$ has potentially good reduction at $p$, then there exists a finite extension $K/\Q$ such that $E/K$ has good reduction at all places of $K$ lying above $p$.\\

Note that the lower bound in the above corollary is independent of the elliptic curve. By Corollary \ref{Elliptic_Case_II}, one may vary over all elliptic curves $E/\Q$ with good reduction at $p$ and still obtain a uniform lower bound on $\widehat{h}$. It is worthwhile comparing this to Lang's conjecture \cite[p. 92]{Lang1}, which states that for a fixed number field $K/\Q$, there exist constants $C_1$ and $C_2$ depending only on $K$ such that for any elliptic curve $E/K$
\begin{equation*}
    \widehat{h}(P) \geq C_1 \, \log N_{K/\Q}\,{\Delta} \, - \, C_2,
\end{equation*}
where $\Delta$ is the discriminant of $E$ over $K$.

\noindent
 
\medskip

\section{\bf A lower bound on local height}\label{local-height}
\bigskip

Let $K$ be a number field. Denote by $M_K$ and $M_K^{\infty}$ the set of all places and the set of archimedean places of $K$ respectively. Suppose that $v \in M_K \setminus M_K^{\infty}$ be a non-archimedean place corresponding to a prime ideal $\rp$ in $K$ above $p$. Denote by $|\cdot|_{v}$ the unique absolute value on $K_{v}$ which extends the usual $p$-adic absolute value on $\Q_p$.\\

Let $E/K$ be an elliptic curve with point of infinity $O$. Similar to Weil height, the canonical height $\widehat{h}$ on $E$ has a decomposition into local height functions $\lambda_{v} : E(K_v) \rightarrow \R$ (called N\'{e}ron local heights) for each place $v$ in $M_K$ (see \cite[Chapter VI, Theorem 1.1]{ATEC}). For any $P \in E(K) \setminus \{ O\}$, we have
\begin{equation}\label{Neron decomposition}
\widehat{h}(P) = \sum_{v \in M_K} \frac{[K_{v} : {\Q}_{v}]}{[K : \Q]} \lambda_{v}(P).
\end{equation}

\medskip

\noindent
Unlike the canonical height $\widehat{h}$, the local height  $\lambda_v$ can take negative values, but such places are only finitely many. We start by discussing explicit formula for $\lambda_v$ and some lower bounds on $\lambda_v$ for both archimedean and non-archimedean places.

\subsection{Archimedean place} Let $v$ be an Archimedean place of $K$. Then, the $v-$adic completion of $K$, $K_v$  is either $\R$ or $\C$. Thus, we can assume $E$ to be defined over $\C$. By uniformization theorem for elliptic curves over $\C$, we get the following isomorphism
$$
E(\C) \hspace{0.5cm} \cong \hspace{0.5cm} \C /\Lambda_{\tau} \hspace{0.5cm} \xrightarrow[z \mapsto u=e^{2 \pi i z}]{\cong} \hspace{0.5cm} \C^* \big/ q_{\tau}^{\Z},
$$
where $\Lambda_{\tau}$ is the lattice $\Z + \tau \Z$ in $\C$ for some $\tau$ in the upper half plane and $q_{\tau} = e^{2 \pi i \tau}$. Then, the local height function
$$
\lambda_v : E(\C) \setminus \{O\} \rightarrow \R
$$
is given by
\begin{equation}\label{Archimedean}
\lambda_v(P) = \, -\frac{1}{2} B_2 \left( \frac{\text{Im}(z)}{\text{Im}(\tau)} \right) \log |q_\tau| - \log|1-u| - \sum_{n \geq 1} \log \left|(1-{q_\tau}^n \, u)(1-{q_\tau}^n \, u^{-1})\right|,
\end{equation}
where $B_2(T) = T^2 - T + \frac{1}{6}$ is the second Bernoulli polynomial for $0 \leq T \leq 1$, extended to $\R$ periodically and $z\,(\text{resp. } u)$ is the image of $P$ in $\C /{\Lambda_\tau}$ (resp. $\C^* \big/{q_{\tau}^{\Z}}$).  The periodicity of $B_2(T)$ ensures that the definition of $\lambda_v(P)$ is independent of the representative $z$ in $\C/\Lambda_\tau$. Further, the convergence of the series is guaranteed by the fact that $\operatorname{Im}(\tau) > 0$. 

\medskip

\noindent


\begin{proposition}\label{Elkies}
    Let $P_1, P_2,\ldots,P_N \in E(\C) \setminus \{O \}$ be $N$ distinct points. Then, there is a constant $b(E) > 0$ depending on $E$ such that
    \begin{equation}\label{Archimedean I}
\sum_{\substack{1 \leq i,\, j \leq N  \\ i \neq j}}\lambda_v(P_i - P_j) \geq -\frac{1}{2} N \log N - b(E) \, N.
    \end{equation}
\end{proposition}

This proposition follows from a result of Elkies, slightly different proof of which is given in \cite[pp. 218]{Hriljack} and \cite[Chapter VI, Theorem 5.1]{Lang}. An explicit version using results from Fourier analysis can be found in \cite[Proposition 4]{Baker-2}.

\subsection{Non-archimedean place}
For a non-archimedean place $v$, let $K_v$ be the completion of $K$. Thus, $E/K$ can be seen as an elliptic curve over $K_v$. Denote by $\Delta$ the discriminant of a Weierstrass equation of $E$. Consider, the minimal Weierstrass equation 
\begin{equation}\label{minimal Weierstrass}
y^2 + a_1xy + a_3y = x^3 + a_2 x^2 + a_4x+ a^6
\end{equation}
of $E$ corresponding to $v$, i.e., $v(\Delta)$ is minimal, subject to the condition that the coefficients are $v$-integral.
Let $\widetilde{E}$ be the curve obtained by reduction of Weierstrass equation of $E$ at $v$. Define,
$$
E_0(K_v) = \left\{ P \in E(K_v) \, | \, \widetilde{E} \text{ is smooth at } \overline{P} \right\},
$$
where $\overline{P}$ is the reduction of $P$ at $v$. Thus, $E_0(K_v)$ consists of all the points of $E(K_v)$ that reduce to a non-singular point after reducing modulo $v$. An explicit formula for $\lambda_v$ at these points is given as follows \cite[Chap. VI, Theorem 4.1]{ATEC}.

\begin{proposition}
    Let $E/K_v$ be an elliptic curve and $v$ be non-archimedean place of $K$. Let $\Delta$ be the discriminant of a Weierstrass equation of $E$ with $v-$integral coefficients. Then the local height function
    \begin{equation}\label{smooth points}
    \lambda_v(P) = \frac{1}{2} \operatorname{max} \left\{ v(x(P)^{-1}), 0 \right\} + \frac{1}{12} v(\Delta)
    \end{equation}
    for all $P \in E_0(K_v) \, \backslash \, \{O\} $.
\end{proposition}

\noindent
The above formula for local height does not depend on the choice of the Weierstrass equation. However, it only holds for points in $E_0(K_v)$, which depends on the choice of Weierstrass equation, and hence, for two different Weierstrass equation, the formula coincides on the intersection. 

\medskip

Moreover, equation \eqref{smooth points} is valid only for points that are non-singular modulo $v$. The explicit description of  $\lambda_v$ on $E(K_v) \, \backslash \, E_0(K_v)$ depends on the nature of reduction type of $E(K_v)$ and is discussed below case by case. For the minimal Weierstrass equation \eqref{minimal Weierstrass}, following \cite[pp. 42]{AEC}, define 
$$
b_2={a_1}^2+ 4 a_4, \,\, b_4=2 a_4+a_1 a_3, \,\, b_6={a_3}^2+4 a_6 \, \text{ and } \, c_4={b_2}^2-24 b_4.
$$

\subsubsection{\bf Additive reduction ($v(\Delta) > 0, \, v(c_4) > 0$) :} 
\hfill

\medskip

Let $f(x) \in K_v[x]$ be a polynomial. For $P \in E(K_v) \setminus \{ O\}$, denote $f(x(P))$ by $f(P)$. 
\medskip

\noindent
Consider the polynomials :
\begin{align*}
F(x) & := 4x^3 + b_2 x^2 + 2b_4x + b_6,\\
G(x) & := (3x^4+b_2 x^3 + 3b_4 x^2 + 3b_6 x +b_8)^2.
\end{align*}

\smallskip

\noindent
Then, for $P \in E(K_v) \, \backslash \, E_0(K_v)$, we have
$$
\lambda_v(P) = \begin{cases}
  -\frac{1}{6} v(F(P)) + \frac{1}{12} v(\Delta) & \text{if } v(G(P)) \geq 3 v(F(P)),\\[0.5em]  
  -\frac{1}{16} v(G(P)) + \frac{1}{12}.
\end{cases}
$$ 

\smallskip

\subsubsection{\bf Multiplicative reduction ($v(\Delta) > 0, v(c_4) = 0$) :} \hfill

\medskip

Let 
$$
\alpha(P) = \operatorname{min} \left\{ \frac{v((2y+a_1x + a_3)P)}{v(\Delta)}, \frac{1}{2} \right\} .
$$
\noindent
Then, for $P \in E(K_v) \, \backslash \, E_0(K_v)$, the local height is given by 
\begin{equation}\label{multiplicative case}
    \lambda_v(P) = \frac{1}{2} B_2(\alpha(P)) v(\Delta).
\end{equation}

\subsubsection{\bf Good reduction $(v(\Delta) = 0$) :} \hfill

\medskip

In this case, we have $E_0(K_{v})= E(K_{v})$ and $v(\Delta) = 0$. Therefore, by \eqref{smooth points}, we get
\begin{equation*}
    \lambda_v(P) = \frac{1}{2} \operatorname{max} \left\{ v(x(P)^{-1}), 0 \right\}    \end{equation*}
for all $P \in E(K_v) \, \backslash \, \{O\}$. Further, the integrality of coefficients of \eqref{minimal Weierstrass} implies that
$$
v(x(P)^{-1}) < 0 \iff v(x(P)/y(P)) < 0.
$$
Using this along with the fact that $
\operatorname{min} \{3 \, v(x(P)), 0\} = \operatorname{min} \{ 2 \, v(y(P)), 0\},
$ we obtain 
$$
\frac{1}{2} \operatorname{max} \{ v(x(P)^{-1}), 0 \} = v(x(P)/y(P)). 
$$
Therefore, in the good reduction case
\begin{equation}\label{Good red}
    \lambda_v(P) = \operatorname{max} \left\{ v\left(\frac{x(P)}{y(P)}\right), 0 \right\}. \end{equation}

\subsection{Explicit formula for local height and Tate map.} Similar to the uniformization theorem for complex elliptic curves, the Tate parametrization offers a way to parametrize elliptic curves over non-archimedean complete fields (\cite[See V.5]{ATEC}). Therefore, for an elliptic curve $E/K_v$, which has multiplicative reduction at $v$, there exists $q_E \in K_v^{\times}$ such that $|q_E|_v < 1$ and the corresponding parametrization (Tate map) is given by the isomorphism
$$
\phi: K_v^{\times} \big/ q_E^{\Z} \xlongrightarrow{\cong} E(K_v) .
$$
    The Tate parametrization provides an analog of explicit formula for the local height function \eqref{Archimedean} in the non-archimedean case as follows (see \cite[Chapter V, Theorem 4.2]{ATEC}).

\begin{thmx}[Tate]\label{Tate-Formula} Let $K_v/\Q_p$ be a finite extension and $E/K_v$ be an elliptic curve with multiplicative reduction at $v$. Let $\phi$ be the Tate map of $E/K_v$ discussed above. Then,
\begin{enumerate}[(i)]
    \item the Néron local height function
$$
\lambda_v \circ \phi: E(K_v) \backslash\{O\} \longrightarrow \mathbb{R}
$$
is given by the formula
$$
\lambda_v(\phi(u))=\frac{1}{2} B_2\left(\frac{v(u)}{v(q_E)}\right) v(q_E)+v(1-u)+\sum_{n \geq 1} v \left(\left(1-q_E^n u\right)\left(1-q_E^n u^{-1}\right)\right) .
$$
\item If we choose $u$ (by periodicity) to satisfy
$$
0 \leq v(u)<v(q_E),
$$
then
$$
\lambda_v(\phi(u))= \begin{cases}\frac{1}{2} B_2\left(\frac{v(u)}{v(q_E)}\right) v(q_E), & \text { if } 0<v(u)<v(q_E), \\ v(1-u)+\frac{1}{12} v(q_E), & \text { if } v(u)=0 .\end{cases}
$$
\end{enumerate}
\end{thmx}

\medskip
\noindent
\textbf{Remark.} Suppose that $E$ has multiplicative reduction. If $P \in E\left(K_v \right)$, we write $u(P)$ for the image of $P$ in $K_{v}^{\times} \big/ q_E^{\Z}$ under the Tate map $E(K_v) \simeq K_v^\times \big/ q_E^{\Z}$. Consider, the retraction homomorphism $r : E\left(K_v\right) \longrightarrow \mathbb{R} / \Z$ defined as
\begin{equation*}\label{retraction}
r(P)=\frac{\log |u(P)|_v}{\log |q_E|_v}.  
\end{equation*}
\noindent
Since $\{ \, |x|_v \,\, | \, \, x \in K_v^\times \, \} \subseteq p^{\mathbb{Q}}$, the image of $r$ is actually contained in the subgroup $\mathbb{Q} / \Z$ of $\mathbb{R} / \Z$.

\medskip

We end this section with the following non-archimedean analog of Proposition \ref{Elkies}.

\begin{proposition}\label{nElkies}
    Let $P_1, P_2,\ldots,P_N \in E(\C_v) \setminus \{O \}$ be $N$ distinct points, where $v$ is a non-archimedean place. Then, 
    \begin{equation}\label{Non-rchimedean II}
\sum_{\substack{1 \leq i,\, j \leq N  \\ i \neq j}}\lambda_v(P_i - P_j) \geq - \frac{N}{12} v({\Delta}).
    \end{equation}
\end{proposition}

Above proposition holds trivially in the good reduction case, so it is enough to consider additive and multiplicative reduction case. Since the local height is invariant under base change to a higher extension, there always exists a finite extension of $K_v$ such that $E$ has either good reduction or multiplicative reduction at $v$. Therefore, it is enough to prove Proposition \ref{nElkies} for the multiplicative reduction case. We need the following result from Fourier analysis for the proof of  Proposition \ref{nElkies} (see \cite{Montgomery} or \cite[Excercise 6.11(b)]{ATEC}).

\begin{lemma}\label{Fourier series}
    Let $t_1, \ldots, t_N \in \mathbb{R}$ and $B_2(t) = (t -[t])^2 - (t-[t]) + \frac{1}{6}$ be the periodic extension of second Bernoulli polynomial on the interval $[0,1]$. Then,
$$
\sum_{\substack{1 \leq i, j \leq N \\ i \neq j}} B_2\left(t_i-t_j\right) \geq-\frac{N}{6} .
$$
\end{lemma}

\begin{proof}
The Fourier expansion of $B_2(t)$ is given by 
$$
B_2(t)=\frac{1}{2 \pi^2} \sum_{n \in \Z, n \neq 0} \frac{e^{2 \pi i n t}}{n^2}.
$$
Therefore, 

\begin{align*}
\sum_{\substack{1 \leq i, j \leq N \\ i \neq j}} B_2\left(t_i-t_j\right) &= \frac{1}{2 \pi^2} \sum_{n \in \Z, n \neq 0} \frac{1}{n^2} \sum_{\substack{1 \leq i, j \leq N \\ i \neq j}} e^{2 \pi i n (t_i - t_j)}\\
&= \frac{1}{2 \pi^2} \sum_{n \in \Z, n \neq 0} \frac{1}{n^2} \left( {\left| \sum_{1 \leq i, j \leq N} e^{2 \pi i n t_i }\,\right|}^2 - N \right)\\
&\geq - \frac{N}{12}.
\end{align*}
\end{proof}

\begin{proof}[Proof of Proposition \ref{nElkies}]
    Since $E/K_v$ has multiplicative reduction, by \eqref{multiplicative case},
    $$
    \sum_{\substack{1 \leq i,\, j \leq N  \\ i \neq j}}\lambda_v(P_i - P_j)  = \frac{v(\Delta)}{2}  \sum_{\substack{1 \leq i,\, j \leq N  \\ i \neq j}} B_2(\alpha(P_i - P_j)). 
    $$
    Applying Lemma \ref{Fourier series}, we have the proposition.
\end{proof}
\medskip 

\section{\bf Proof of Theorem \ref{Elliptic Case I}}\label{proof-elliptic}

\bigskip 

Let $E/K$ be an elliptic curve over a number field $K$ and $\rL = \bigcup_i L_i $ be an asymptotically positive tower over $K$. Let $L \subseteq \rL$ be a finite extension of $K$. For a finite set $Z =\{ P_1, P_2, \ldots, P_N \} \subseteq E(L)$, define 
\begin{equation*}
    \widehat{h}(Z) = \frac{1}{|Z|} \sum_{P\in Z} \widehat{h}(P)
\end{equation*}
to be the average of canonical heights. Using the parallelogram property of canonical height, we obtain
\begin{align}\label{Decomposition formula}
 \widehat{h}(Z) &\geq \frac{1}{4N(N-1)} \sum_{\substack{1 \leq i,\, j \leq N  \\ i \neq j}} \widehat{h}(P_i-P_j) \nonumber \\ 
 &\geq  \frac{1}{4N(N-1)} \mathlarger{\mathlarger{\sum}}_{v \in M_L} \frac{[L_v : \Q_{v}]}{[L:\Q]} \sum_{\substack{1 \leq i,\, j \leq N  \\ i \neq j}} \lambda_v(P_i - P_j).
 \end{align}

Now we begin the proof of Theorem \ref{Elliptic Case I}. Throughout the proof, we assume that $E/L$ has either good or split multiplicative reduction at any non-archimedean place $v \in M_L$. If not, this can be ensured by a finite base change. The proof consists of three parts: obtaining lower bound of the contribution in \eqref{Decomposition formula} due to archimedean places, places of good reduction, and places of split multiplicative reduction. The lower bound for the split multiplicative case is inspired by \cite[Theorem 20]{Baker-2}.




\begin{proof}[\bf Proof of Theorem \ref{Elliptic Case I}]

For an archimedean place $ v \in M_L^{\infty}$, from \eqref{Archimedean I} we know that
\begin{equation*}
 \sum_{\substack{1 \leq i,\, j \leq N  \\ i \neq j}} \lambda_v(P_i - P_j) \geq -\frac{1}{2} N  \log N - b(E) N.
\end{equation*}
\noindent
Since $\sum_{v \in M_L^{\infty}} {[L_v : \Q_{v}]}/{[L : \Q]}=1$, we can write 
\begin{equation}\label{archimedean places contribution}
\mathlarger{\mathlarger{\sum}}_{v \in M_L^{\infty}} \frac{[L_v : \Q_{v}]}{[L : \Q]} \sum_{\substack{1 \leq i,\, j \leq N  \\ i \neq j}} \lambda_v(P_i - P_j) \geq -\frac{1}{2} N  \log N - b(E) N.  
\end{equation}
\medskip

\noindent
Suppose $E$ has good reduction at $v \in M_L \setminus M_L^{\infty}$, a non-archimedean place lying above the rational prime $p$. By \eqref{Good red}, we have

$$
\sum_{\substack{1 \leq i, \, j \leq N \\ i \neq j}} \lambda_v(P_i - P_j) \geq \frac{1}{2} \sum_{\substack{1 \leq i, \, j \leq N \\ i \neq j}}  \operatorname{max} \left\{ v\left(\frac{x(P_i - P_j)}{y(P_i - P_j)}\right), 0 \right\}.
$$  
Let $\mathbb{F}_q$ be the residue field of $L_v / \Q_p$ and $\widetilde{E}$ be the reduction of $E$ at $v$. Suppose 
 $$
    \widetilde{E}(\mathbb{F}_q) = \{ T_1, T_2, \ldots, T_m\}.
 $$ 
 For $1 \leq l \leq m$, define 
 
 \begin{equation*}
     A_l := \left| \,\left\{ \, P_i \in Z \,|\, P_i \equiv T_l \, (\,\operatorname{mod} v \,) \right\} \right|.
 \end{equation*}
 Clearly, $\sum_{i=1}^m A_i = N$. Note that if $P_i \equiv P_j \mod v $, then $$v\left(\frac{x(P_i - P_j)}{y(P_i - P_j)}\right) \geq \frac{\log p}{e_v},$$ 
 where $p$ is the rational prime dividing $v$ and $e_v$ is ramification index of $v$. Therefore,
 \begin{align*}
 \sum_{\substack{1 \leq i, \, j \leq N \\ i \neq j}} \lambda_v(P_i - P_j) &\geq \frac{1}{2} \frac{\log p}{e_v} \left( \sum_{l=1}^m  A_l \, (A_l-1) \right)\\
 &\geq \frac{1}{2} \frac{\log p}{e_v}\left( \frac{N^2}m - N \right),
\end{align*}
where the last step follows from the Cauchy-Schwartz inequality. By Hasse's inequality,
 $$
 m \leq q + 1 + 2 \sqrt{q}.
 $$
Hence,
 $$
 \sum_{\substack{1 \leq i, \, j \leq N \\ i \neq j}} \lambda_v(P_i - P_j) \geq \frac{1}{2[L_v : \Q_p]} \left( \frac{N^2}{q + 1 + 2 \sqrt{q}} - N \right) \log q.
$$

\noindent
Therefore,  
\begin{equation}\label{good reduction contribution}
\sideset{}{'}\sum_{v \in M_L} \frac{[L_v : \Q_p]}{[L:\Q]}  \sum_{\substack{1 \leq i, \, j \leq N \\ i \neq j}} \lambda_v(P_i - P_j) \geq \frac{N^2}{2} \, \mathlarger{\mathlarger{\sum}}_{q} \, \, \frac{G_q(L)}{[L:\Q]} \left(\frac{1}{q+ 1+ 2\sqrt{q}} - \frac{1}{N}\right) \log q,
\end{equation}
\noindent
where $\sideset{}{'}\sum$ runs over all non-archimedean $v \in M_L$ at which $E$ has good reduction.\\

\noindent
Finally, suppose $E$ has split multiplicative reduction at $v \in M_L \setminus M_L^{\infty}$, a non-archimedean place with norm $q$. Let $k_v=-\operatorname{ord}_v \left(j_{E}\right) \geq 1$. Recall that the retraction homomorphism $r: E\left(K_v \right) \rightarrow \R / \Z$ factors as
$$
E\left(K_v\right) \longrightarrow K_v^\times / q_E^{\Z} \longrightarrow \R / \Z,
$$
where $q_E \in K_v^\times$ with $|q_E|_v=\left|1 / j_E \right|_v$. Here the first map is the Tate parametrization and the second map is given by $u \mapsto \frac{\log |u|_{v}}{ \log |q|_{v}}$. Note that $\operatorname{Im}(r)=\left\langle 1 / k_v \right\rangle \subseteq$ $\R / \Z$, and therefore $e^{2 \pi i n r\left(P_j\right)}=1$ whenever $k_v \mid n$. Since $P_i \neq P_j$ for $i \neq j$, using Theorem \ref{Tate-Formula} (ii), we obtain
\begin{align*}
\sum_{\substack{1 \leq i, \, j \leq N \\ i \neq j}} \lambda_v(P_i - P_j) & \geq \frac{1}{2}  \ v({q}_E) \sum_{\substack{1 \leq i, \, j \leq N \\ i \neq j}} B_2 \left( \frac{v(u(P_i)) - v(u(P_j))}{v(q_E)} \right)  \\
& \geq \frac{1}{2} \ v(q_E) \left(\frac{1}{2 \pi^2} \sum_{n \in \Z, \ n \neq 0} \frac{1}{n^2} \sum_{\substack{1 \leq i, j \leq N \\ i \neq j}} e^{2 \pi i n \left( \frac{v(u(P_i)) - v(u(P_j))}{v(q_E)}\right)} \right) \\
& \geq \frac{1}{4 \pi^2} \ v(q_E) \left( N(N-1) \sum_{n \in k_v \Z, \ n \neq 0} \frac{1}{n^2} +  \sum_{n \in \Z \setminus{k_v \Z}} \frac{1}{n^2}\sum_{\substack{1 \leq i, j \leq N \\ i \neq j}} e^{2 \pi i n \left( \frac{v(u(P_i)) - v(u(P_j))}{v(q_E)}\right)} \right)
\\
& \geq \frac{1}{4 \pi^2} \ v(q_E) \left( N(N-1) \sum_{n \in k_v \Z, \ n \neq 0} \frac{1}{n^2} +  \sum_{n \in \Z \setminus{k_v \Z}} \frac{1}{n^2} \left( \left\mid \sum_{\substack{1 \leq i \leq N }}  e^{2 \pi i n \left( \frac{v(u(P_i)) }{v(q_E)}\right)} \right\mid^2 - N \right) \right)
\\
& \geq N(N-1) \frac{\log \left|j_E\right|_v}{4 \pi^2} \sum_{n \in k_v \Z \backslash\{0\}} \frac{1}{n^2} + O(N) 
\end{align*}
As $\zeta(2) = \sum_{n\geq 1} \frac{1}{n^2} = \frac{\pi^2}{6}$, we deduce that
\begin{align*}
\sum_{\substack{1 \leq i, \, j \leq N \\ i \neq j}} \lambda_v(P_i - P_j) & \geq N(N-1) \frac{\log \left|j_E\right|_v}{12 \, {k_v}^2} + O(N) \\
& \geq N(N-1) \frac{\log p}{12 \, k_v} + O(N).
\end{align*}

\noindent
Since $j_E$ is fixed, there exists a constant $c_E > 0$ only dependent on $E$ such that 
 $$
k_v  \leq \, \frac{ e_v}{c_E},
$$
where $e_{v}$ is the ramification index of $L_{v}/\Q_p$. Therefore,
 \begin{align}\label{split mul contribution}
 \sideset{}{''}\sum_{v \in M_L} \frac{[L_v : \Q_p]}{[L:\Q]}  \sum_{\substack{1 \leq i, \, j \leq N \nonumber \\ i \neq j}} \lambda_v(P_i - P_j)  \geq &  \sideset{}{''}\sum_{ \, v \in M_L}  \frac{[L_v : \Q_p]}{[L:\Q]} \left( \frac{\log p}{12 \, k_v} N(N-1)  + O(N)\right) \nonumber \\
\geq & \, c_E \, \sideset{}{''}\sum_{v \in M_L}  \frac{[L_v : \Q_p]}{[L:\Q]} \left( \frac{\log p}{12 \,  e_v} N(N-1)  + O(N)\right) \nonumber\\ 
 \geq &  \,  \, c_E \, N(N-1) \sum_{q} \frac{B_{q}(L)}{[L:\Q]} \frac{\log q}{12} + O(N),
 \end{align}
 where $\sideset{}{''}\sum$ runs over all non-archimedean $v \in M_L$ at which $E/L$ has split multiplicative reduction and $q$ runs over all prime powers. Combining \eqref{archimedean places contribution}, \eqref{good reduction contribution} and \eqref{split mul contribution} in \eqref{Decomposition formula}, we obtain  
\begin{align}\label{Final}
\widehat{h}(Z)   \geq \frac{1}{4\,N(N-1)} &\biggl( -\frac{1}{2} N  \log N - b(E) N + \frac{N^2}{2} \, \mathlarger{\mathlarger{\sum}}_{q} \, \, \frac{G_q(L)}{[L:\Q]} \left(\frac{1}{q+ 1+ 2\sqrt{q}} - \frac{1}{N}\right) \log q \nonumber \\
& \, +
N(N-1) \, \mathlarger{\mathlarger{\sum}}_{q} \, \, \frac{B_{q}(L)}{[L:\Q]} \frac{\log q}{12} + O(N) \biggl).    
\end{align}

For any $N$, the above inequality holds for all possible $N$-points $\{P_1,P_2, \cdots,  P_N\} \subset E(L)$. Considering an infinite sequence $\{P_n\}_{n \geq 1}$ in $E(\rL)$ such that $\lim \hat{h}(P_n) = \liminf_{P \in E(\rL)} \hat{h}(P)$. Thus $\lim \hat{h}(Z_N) = \liminf_{P \in E(\rL)} \hat{h}(P)$, where $Z_N = \{P_1, P_2, \ldots, P_N\}$. Taking $N$ large, \eqref{Final} implies that
    \begin{align*}
        \liminf_{P\in E(\rL)}\, \widehat{h}(P) \geq & \,  \frac{1}{8}  \,\mathlarger{\mathlarger{\sum}}_{q} \, \,  \xi_q(\rL) \, \frac{\log q}{q+ 1 + 2\sqrt{q}} +  \frac{c_E}{48} \, \mathlarger{\mathlarger{\sum}}_{q} \, \,  \chi_q(\rL) \, \log q
    \end{align*}
    as required.
\end{proof}

\medskip 

\section{\bf Concluding remarks}
\bigskip

The notion of relative height as discussed in Section \ref{relative-height} leads to several interesting questions. For instance, one can formulate a stronger Schinzel-Zassenhaus conjecture in this context. In \cite{SZ}, A. Schinzel and H. Zassenhaus proposed a weaker version of Lehmer's conjecture, which was recently resolved by V. Dimitrov \cite{Dimitrov}. Let $\house{\strut \alpha}$ denote the house of $\alpha \in \overline{\Q}$ defined as
$$
   \house{\strut \alpha} = \max_i \{|\alpha_i|\},
$$
where $\alpha_i$'s are the conjugates of $\alpha$. Then Dimitrov's theorem states that for any non-zero $\alpha\in \overline{\Q}$ not a root of unity,
\begin{equation*}
    \log \house{\strut\alpha} > \frac{c}{[\Q(\alpha): \Q]}
\end{equation*}
for $c = (\log 2)/4$. In the context of relative heights, it is possible to formulate a generalization of this conjecture in the following way. Let $m$ be a fixed positive integer. For any algebraic integer $\alpha \in \overline{\Q}$ define
\begin{equation*}
     \house{\strut \alpha}_m := \max_{\substack{[\Q(\alpha):K]\leq m\\ \sigma:K \hookrightarrow \C}} M_{\sigma(K)} (\alpha). 
\end{equation*}
When $m=1$, it is clear that $\house{\strut \alpha}_1 = \house{\strut \alpha}$. Thus, it is reasonable to make the following conjecture, in the spirit of Dimitrov's theorem.
\begin{conjecture}
    Let $m$ be a fixed positive integer. For all non-zero algebraic integer $\alpha \in \overline{\Q}$, which are not roots of unity,
    \begin{equation*}
        \log \house{\strut\alpha}_m > \frac{c}{[\Q(\alpha): \Q]},
    \end{equation*}
    for an absolute constant $c>0$.
\end{conjecture}

\bigskip

\section*{\bf Acknowledgments}
We are grateful to Prof. Sinnou David for several insightful comments, which simplified our proof of Theorem \ref{Elliptic Case I} avoiding the theory of discrepancies. We are thankful to Prof. Martin Widmer, Prof. Lukas Pottmeyer and Prof. Michel Waldschmidt for  fruitful discussions and encouragement. We thank Prof. Sara Checcoli and Prof. Arno Fehm for their comments on an earlier version of this paper. We thank Prof. V. Kumar Murty, Prof. M. Ram Murty, Prof. Arnaud Plessis, Prof. Ananth Shankar, Prof. Siddhi Pathak and Jenvrin Jonathan for several helpful suggestions. 





\bigskip 

\end{document}